\documentclass{amsart}
\usepackage{amssymb,amsfonts}
\usepackage[all,arc]{xy}
\usepackage{enumerate}
\usepackage{mathrsfs}
\usepackage{hyperref}
\usepackage{wasysym}
\usepackage{mathtools}
\usepackage{setspace}
\usepackage[style=alphabetic,sorting=nyt]{biblatex}
\newtheorem{thm}{Theorem}[section]
\newtheorem{cor}[thm]{Corollary}
\newtheorem{prop}[thm]{Proposition}
\newtheorem{lem}[thm]{Lemma}

\theoremstyle{definition}
\newtheorem{defn}[thm]{Definition}
\newtheorem{defns}[thm]{Definitions}

\newtheorem{exmp}[thm]{Example}

\newtheorem{rem}[thm]{Remark}

\newtheorem{assump}[thm]{Assumption}

\newcommand{\mycomment}[1]{}
\newcommand{\cal}{\mathcal}
\newcommand{\bb}{\mathbb}

\DeclareMathOperator{\Spec}{Spec}
\DeclareMathOperator{\Spf}{Spf}
\DeclareMathOperator{\Proj}{Proj}
\DeclareMathOperator{\Pic}{Pic}

\DeclareMathOperator{\Cl}{Cl}

\DeclareMathOperator{\chara}{char}
\DeclareMathOperator{\Aut}{Aut}

\DeclareMathOperator{\et}{\acute{e}t}

\allowdisplaybreaks

\makeatletter
\let\c@equation\c@thm
\makeatother
\numberwithin{equation}{section}

\title{Extendability of group actions on K3 or Enriques surfaces}

\author{Tianchen Zhao}
\begin{document}

\begin{abstract}
    Let $X$ be a K3 or Enriques surface with good reduction. Let $G$ be a finite group acting (not necessarily linearly) on $X$. We give a criterion for this group action to extend to a smooth model of $X$ in terms of the action of $G$ on the second $\ell$-adic cohomology groups. In particular, we generalize the result on the extendability of Galois actions on K3 surfaces by Chiarellotto, Lazda, and Liedtke. As an application, we prove that a symplectic linear group action is extendable if the residue characteristic does not divide its order. Lastly, we relate the good reduction of Enriques surfaces with that of their K3 double covers.
\end{abstract}

\maketitle

\setcounter{tocdepth}{1}
\tableofcontents

\section{Introduction}

Throughout the whole article, we fix the following notations.
$$\begin{array}{ll}
    K & \mbox{a complete and discretely valued field} \\
    \cal{O}_K & \mbox{the ring of integers of $K$} \\
    k & \mbox{the residue field of $K$, which we assume to be perfect of characteristic $p$} \\
    K^{\mathrm{s}} & \mbox{a fixed separable closure of $K$} \\
    \overline{k} & \mbox{the residue field of $K^{\mathrm{s}}$, an algebraic closure of $k$} \\
    G_{E/F} & \mbox{the Galois group of $E/F$, a Galois extension} \\
    I_{E/F} & \mbox{the inertia group of $E/F$, a Galois extension of complete and} \\ & \mbox{discretely valued fields} \\
    G_F & \mbox{an absolute Galois group of a field $F$} \\
    I_F & \mbox{an absolute inertia group of a complete and discretely valued field $F$}
\end{array}$$

\begin{defns}
    Let $X$ be a smooth proper variety over $K$.
    \begin{itemize}
        \item A model of $X$ is a proper flat algebraic space over $\cal{O}_K$ whose generic fibre is isomorphic to $X$.
        \item We say $X$ has good reduction if $X$ has a smooth model. Also, we say $X$ has potential good reduction if $X_L$ has good reduction for some finite extension $L/K$.
    \end{itemize}
\end{defns}

In this article, we study when a (not necessarily linear) group action on a K3 or Enriques surface extends to a smooth model. In particular, any automorphisms and group actions are not assumed to be linear over the ground field unless explicitly stated otherwise.

\begin{defn}
    Let $X$ be a smooth proper variety over $K$ with good reduction. Let $G$ be a group acting on $X$. We say the $G$-action on $X$ is extendable if the $G$-action extends to some smooth model $\cal{X}$ of $X$.
\end{defn}

Note that we only require the action to extend to some smooth model rather than a particular smooth model.

\begin{exmp}
    Let $A$ be an abelian variety over $K$ with good reduction. Every smooth model of $A$ over $
    \cal{O}_K$ is an abelian scheme, and therefore coincides with the N\'eron model by \cite[Corollary 1.4]{Neron_model_Artin}. Then every group action on $A$ is extendable by the N\'eron mapping property.
\end{exmp}

\subsection{Good reduction and the extendability problem}

The extendability of group actions plays an important role in the good reduction problem.

In \cite{liedtke_2017_good_reduction_k3surfaces}, Liedtke and Matsumoto prove that if a K3 surface $X$ over $K$ has good reduction after a finite separable extension $L/K$ and $H^2_{\et}(X_{K^{\text{s}}},\bb{Q}_\ell)$ is unramified, then $X$ has good reduction after a finite and unramified extension. The idea of the proof is roughly as follows. By passing to the Galois closure, we assume $L/K$ is Galois. One finds a smooth model $\cal{X}$ of $X_L$ such that the action of $I_{L/K}$ on $X_L$ extends to $\cal{X}$. Then $\cal{X}/I_{L/K}$ is a smooth model of $X$ over the maximal unramified sub-extension of $L/K$.

Another example is in \cite{Chiarellotto_2019_neron_ogg_shafarevich_criterion_k3surface}. Let $X$ be a K3 surface over $K$ with good reduction after a finite and unramified extension $L/K$. Again, we may assume $L/K$ is Galois. Chiarellotto, Lazda, and Liedtke derive a criterion for the extendability of the action of $G_{L/K}$ on $X_L$. If the action does extend to some smooth $\cal{X}$ of $X_L$, then $\cal{X}/G_{L/K}$ gives a smooth model of $X$.

In addition to the extendability of Galois actions, the extendability of linear group actions also has applications to the good reduction problem. See Corollary \ref{good reduction of enriques surfaces} for an example.

\subsection{Main results}

We focus on the case where $X$ is a K3 or Enriques surface. In particular, for any smooth model $\cal{X}$ of $X$, its special fibre $\cal{X}_k$ is determined by $X$ up to a canonical isomorphism by \cite[Proposition 4.7]{liedtke_2017_good_reduction_k3surfaces}. We call it the canonical reduction of $X$, denoted by $X^\circ$. This is also a K3 or Enriques surface depending on $X$. In fact, the canonical reduction of $X$ is defined as long as $X$ has good reduction after a finite unramified extension, as explained in Theorem \ref{RDP models}.

Moreover, there is a canonical homomorphism
    $$\mathrm{sp}:\Aut(X)\to\Aut(X^\circ)\ \ \ g\mapsto g_k,$$
called the specialization homomorphism. Any group action on $X$ induces a group action on $X^\circ$ via specialization. See Section \ref{specialization homomorphism subsection} for more details.

Let $G$ be a finite group acting on $X$ on the right. This induces a $G$-action on $K$ on the left whose fixed field is $K_0=K^G$ (we call this the base field of the $G$-action as in Definition \ref{base field}). We restrict to the case where $K/K_0$ is finite and unramified.

\begin{rem}
    We will see why we want $G$ to act on the right in Lemma \ref{ugly cocycle lemma}. This assumption is merely a formality. Indeed, suppose $G$ acts on $X$ on the left via $\rho:G\to \Aut(X)$. We can convert this to a right action by taking the anti-homomorphism $\rho':G\to \Aut(X)$, $g\mapsto \rho(g^{-1})$.
\end{rem}

Let $\ell\neq p$ be a prime number. We want to give a criterion for the extendability of group actions in terms of the $\ell$-adic cohomology of $X$ and $X^\circ$. Since $X$ has good reduction, $I_K$ acts trivially on $H^2_{\et}(X_{K^{\mathrm{s}}},\bb{Q}_\ell(1))$. Then $G_k\cong G_K/I_K$ acts on $H^2_{\et}(X_{K^{\mathrm{s}}},\bb{Q}_\ell(1))$. Denote by $H^2_{\et}(X_{K^{\mathrm{s}}},\bb{Q}_\ell(1))^{G_k}$ and $H^2_{\et}(X^\circ_{\overline{k}},\bb{Q}_\ell(1))^{G_k}$ respectively the $G_k$-invariant subspaces of $H^2_{\et}(X_{K^{\mathrm{s}}},\bb{Q}_\ell(1))$ and $H^2_{\et}(X^\circ_{\overline{k}},\bb{Q}_\ell(1))$. As in Section \ref{fibre product action}, we can define a $G$-action on $H^2_{\et}(X_{K^{\mathrm{s}}},\bb{Q}_\ell(1))^{G_k}$ and similarly on $H^2_{\et}(X^\circ_{\overline{k}},\bb{Q}_\ell(1))^{G_k}$ via specialization.

By adapting the method in \cite{Chiarellotto_2019_neron_ogg_shafarevich_criterion_k3surface}, we prove the following.

\begin{thm}\label{main theorem}
    Let $X$ be a K3 or Enriques surface over $K$ with good reduction. Let $G$ be a finite group acting on $X$ with base field $K_0$. Assume $K/K_0$ is finite and unramified. Then the $G$-action on $X$ is extendable if and only if there is a $G$-isomorphism $$H^2_{\et}(X_{K^{\mathrm{s}}},\bb{Q}_\ell(1))^{G_k}\cong H^2_{\et}(X^\circ_{\overline{k}},\bb{Q}_\ell(1))^{G_k}.$$
\end{thm}

A further criterion is given in Theorem \ref{refined main theorem}. In Section \ref{fibre product action}, we define a group $\overline{G}_k$ that acts on $X_{K^{\mathrm{s}}}$ and $X_{\overline{k}}$, which combines the action of $G$ on $X$ and the Galois action on $K^{\mathrm{s}}$ and $\overline{k}$. Then we show that the condition in Theorem \ref{main theorem} is further equivalent to the existence of a $\overline{G}_k$-isomorphism $H^2_{\et}(X_{K^{\mathrm{s}}},\bb{Q}_\ell(1))\cong H^2_{\et}(X^\circ_{\overline{k}},\bb{Q}_\ell(1))$.

Theorem \ref{main theorem} recovers the criterion for the extendability of Galois action in \cite{Chiarellotto_2019_neron_ogg_shafarevich_criterion_k3surface}. Let $L/K$ be a finite and unramified extension. Let $k_L$ be the residue field of $L$. Let $X$ be a K3 surface over $K$ such that $X_L$ has good reduction. By passing to the Galois closure, we may assume $L/K$ is Galois.

\begin{cor}[{\cite[Theorem 1.6]{Chiarellotto_2019_neron_ogg_shafarevich_criterion_k3surface}}]
    The following are equivalent.\begin{enumerate}
        \item $X$ has good reduction.
        \item The action of $G_{L/K}$ on $X_L$ is extendable.
        \item There is a $G_K$-isomorphism $H^2_{\et}(X_{K^{\text{s}}},\bb{Q}_{\ell})\cong H^2_{\et}(X_{\overline{k}}^\circ,\bb{Q}_{\ell})$.
        \item There is a $G_{L/K}$-isomorphism $H^2_{\et}(X_{K^{\text{s}}},\bb{Q}_{\ell}(1))^{G_{k_L}}\cong H^2_{\et}(X_{\overline{k}}^\circ,\bb{Q}_{\ell}(1))^{G_{k_L}}$.
    \end{enumerate}
\end{cor}
\begin{proof}
    Suppose the $G_{L/K}$-action on $X_L$ extends to a smooth model $\cal{X}$. Then $\cal{X}/G_{L/K}$ is a smooth model of $X$ by \cite[Proposition 7.1]{Chiarellotto_2019_neron_ogg_shafarevich_criterion_k3surface}. This proves $(2)\implies(1)$.

    Next, suppose $X$ has a smooth model $\cal{X}$. Then the comparison isomorphism $H^2_{\et}(X_{K^{\mathrm{s}}},\bb{Q}_{\ell})\cong H^2_{\et}(X_{\overline{k}}^\circ,\bb{Q}_{\ell})$ is $G_K$-equivariant. This proves $(1)\implies(3)$.

    For $(3)\implies(4)$, we simply restrict the $G_K$-isomorphism to the $G_{k_L}$-invariant subspace. Finally, $(4)\implies(2)$ follows by Theorem \ref{main theorem}.
\end{proof}

In case of a $K$-linear action, $G$ acts on $H^2_{\et}(X_{K^{\mathrm{s}}},\bb{Q}_\ell)$ and on $H^2_{\et}(X^\circ_{\overline{k}},\bb{Q}_\ell)$ via specialization. We have the following criterion for extendability.

\begin{thm}\label{linear case main theorem}
    Let $X$ be a K3 or Enriques surface over $K$ with good reduction. Let $G$ be a finite group acting $K$-linearly on $X$. Then the $G$-action on $X$ is extendable if and only if there is a $G$-isomorphism $$H^2_{\et}(X_{K^{\mathrm{s}}},\bb{Q}_{\ell})\cong H^2_{\et}(X^\circ_{\overline{k}},\bb{Q}_{\ell}).$$
\end{thm}

In particular, this implies that the extendability of $K$-linear group actions is a geometric property rather than an arithmetic property.

\subsection{Applications of the main results}

We present some applications of our main results. We will prove them in Section \ref{application}.

\subsubsection{Extendability of symplectic automorphisms}

The first application is a criterion for the extendability of a symplectic group action in terms of the order of the group and the residue characteristic.

\begin{defn}
    Let $X$ be a K3 surface over a field $F$. An $F$-linear automorphism of $X$ is symplectic if the induced automorphism of $H^0(X,\omega_X)$ is the identity. An $F$-linear group action of $G$ on $X$ is symplectic if for every $g\in G$, the automorphism $g:X\to X$ is symplectic.
\end{defn}

\begin{cor}\label{characteristic and extendability}
    Let $X$ be a K3 surface over $K$ with good reduction. Let $G$ be a finite group acting linearly and symplectically on $X$. If $p$ does not divide the order of $G$, then the $G$-action is extendable.
\end{cor}
\begin{rem}
    Matsumoto proves this in \cite[Theorem 1.1]{matsumoto_2021_extendability_automorphisms_k3surfaces} with a weaker notion of extendability. In \cite{matsumoto_2021_extendability_automorphisms_k3surfaces}, an automorphism $g$ of a K3 surface $X$ with good reduction is defined to be extendable if there exists a finite extension $L/K$ such that $g$ extends to a smooth model of $X_L$. This corresponds to potential extendability in our terminology.
\end{rem}
\begin{rem}
    Our proof of Corollary \ref{characteristic and extendability} is different from that of \cite[Theorem 1.1]{matsumoto_2021_extendability_automorphisms_k3surfaces}. Despite this, we need \cite[Lemma 2.13]{matsumoto_2021_extendability_automorphisms_k3surfaces} and \cite[Proposition 6.1]{matsumoto_2021_extendability_automorphisms_k3surfaces}. Neither of them depend on the main result of \cite{matsumoto_2021_extendability_automorphisms_k3surfaces}. Indeed, \cite[Lemma 2.13]{matsumoto_2021_extendability_automorphisms_k3surfaces} is a general fact about K3 surfaces and \cite[Proposition 6.1]{matsumoto_2021_extendability_automorphisms_k3surfaces} is based on a mild generalization of the flops considered in \cite[Section 4]{liedtke_2017_good_reduction_k3surfaces}.
\end{rem}

\subsubsection{Good reduction of Enriques surfaces}

We can also relate the good reduction of Enriques surfaces with that of their K3 double covers. Recall that for every Enriques surface $Y$ over a field $F$ of characteristic different from $2$, there exists a K3 surface $X$ with a fixed-point-free involution $\iota$ such that $Y=X/\iota$. We say this $X$ is a K3 double cover of $Y$ and $\iota$ is the associated involution.

\begin{cor}\label{good reduction of enriques surfaces}
    Assume $p\neq 2$. Let $Y$ be an Enriques surface over $K$. Then the following are equivalent. \begin{enumerate}
        \item $Y$ has good reduction.
        \item There exists a K3 double cover $X$ of $Y$, with associated involution $\iota$, such that $X$ has good reduction and $\iota_k$ has no fixed points.
        \item There exists a K3 double cover $X$ of $Y$, with associated involution $\iota$, such that $X$ has good reduction and $\dim H^2_{\et}(X_{\overline{k}}^\circ,\bb{Q}_{\ell})^{\iota_k=\mathrm{id}}=10$.
    \end{enumerate}
\end{cor}

\noindent \textbf{Acknowledgements.} I would like to express my sincere gratitude to my supervisor, Christopher Lazda, for his invaluable support in many aspects. I also want to thank Yuya Matsumoto for helpful correspondence, and the anonymous referee for many helpful suggestions. This work would not have been possible without their help. Last but not least, I gratefully acknowledge the financial support from the University of Exeter.

\section{Preliminaries}

In this section, we review the preliminaries of this article, including canonical reduction, RDP models, Weyl groups, the specialization homomorphism, and group cohomology with coefficients in a non-abelian group.

We first remark on the characteristic of $K$:

\begin{rem}
    Several results we cite from \cite{liedtke_2017_good_reduction_k3surfaces}, \cite{Chiarellotto_2019_neron_ogg_shafarevich_criterion_k3surface} and \cite{matsumoto_2021_extendability_automorphisms_k3surfaces} are stated under the assumption $\chara K=0$. This hypothesis is not needed for the results we use:
    \begin{enumerate}
        \item As pointed out in \cite{good_reduction_k3_equichara}, the results of \cite{liedtke_2017_good_reduction_k3surfaces} and \cite{Chiarellotto_2019_neron_ogg_shafarevich_criterion_k3surface} in the $\ell$-adic case are valid with the same proof even if $\chara K>0$ as long as $X$ has good reduction after a finite separable extension.
        \item \cite[Proposition 2.2, Proposition 4.6, Proposition 6.1]{matsumoto_2021_extendability_automorphisms_k3surfaces} hold since they build on \cite{liedtke_2017_good_reduction_k3surfaces}. \cite[Lemma 2.13]{matsumoto_2021_extendability_automorphisms_k3surfaces} is a general result about K3 surfaces whose proof also settles the case $\chara K>0$.
    \end{enumerate}
\end{rem}

\subsection{RDP models, canonical reductions, and Weyl groups}\label{subsection RDP models and Weyl groups}

We define RDP models and canonical reduction of K3 and Enriques surfaces in this sub-section. To motivate the definition of RDP models, let $X$ be a K3 or Enriques surface over $K$ with good reduction. There is no canonical choice of smooth models of $X$. However, once a polarization $\cal{L}$ is fixed, $X$ admits a canonical model, called the RDP model, that may admit mild singularities and only depends on $X$ and $\cal{L}$. In fact, the RDP model can be defined as long as $X$ has good reduction after a finite unramified extension.

\begin{thm}[\cite{liedtke_2017_good_reduction_k3surfaces}]\label{RDP models}
    Let $X$ be a K3 or Enriques surface over $K$ with good reduction after a finite unramified extension. Let $\cal{L}$ be a polarization of $X$. Then there exists a projective model $P(X,\cal{L})$ of $X$, called the RDP model, such that the following hold:
\begin{enumerate}
    \item $P(X,\cal{L})$ only depends on the pair $(X,\cal{L})$ up to a canonical isomorphism.
    \item $P(X,\cal{L})_k$ has at worst rational double points.
    \item The minimal desingularization $X^\circ$ of $P(X,\cal{L})_k$ is a K3 or Enriques surface respectively and it is determined by $X$ up to a canonical isomorphism.
\end{enumerate}
\end{thm}

\begin{defn}
    We say $X^\circ$ is the canonical reduction of $X$.
\end{defn}

To prove Theorem \ref{RDP models}, we need the notion of $\cal{L}$-terminal models.

\begin{defn}
    Let $X$ be a K3 or Enriques surface over $K$ with good reduction. Let $\cal{L}$ be a polarization of $X$. We say a smooth model $\cal{X}$ of $X$ is $\cal{L}$-terminal if the specialization $\cal{L}_k$ of $\cal{L}$ to $\cal{X}_k$ is big and nef. Equivalently, we say $(\cal{X},\cal{L})$ is terminal.
\end{defn}

\begin{rem}\label{existence of L-terminal model}
    By \cite[Proposition 4.5]{liedtke_2017_good_reduction_k3surfaces}, a K3 or Enriques surface with good reduction always admits an $\cal{L}$-terminal model for any polarization $\cal{L}$.
\end{rem}

\begin{proof}[Proof of Theorem \ref{RDP models}]
    This proof is from \cite[Section 6]{Chiarellotto_2019_neron_ogg_shafarevich_criterion_k3surface}.
    Let $X$ and $\cal{L}$ be as in Theorem \ref{RDP models}. Let $L/K$ be a finite and unramified Galois extension such that $X_L$ has good reduction. As in Remark \ref{existence of L-terminal model}, $X_L$ admits an $\cal{L}_L$-terminal model $\cal{Y}$. Since the specialization of $\cal{L}_L$ to $\cal{Y}_{k_L}$ is big and nef, we obtain a proper birational morphism
    \[
        \pi:\cal{Y}\to P(X_L,\cal{L}_L):=\Proj \bigoplus_n H^0(\cal{Y},\cal{L}_L^{\otimes n}).
    \]
    Since $\cal{L}_L$ is ample on $X_L$, the generic fibre of $\pi$ is an isomorphism, i.e. $P(X_L,\cal{L}_L)$ is a model of $X_L$. Also, by \cite[Proposition 4.6]{liedtke_2017_good_reduction_k3surfaces}, $P(X_L,\cal{L}_L)_{k_L}$ is a proper surface with at worst rational double points, and the special fibre of $\pi$ is its minimal desingularization. In particular, the minimal desingularization of $P(X_L,\cal{L}_L)_{k_L}$ is a K3 or Enriques surface respectively.
    
    Next, the $G_{L/K}$-action on $X_L$ extends to $P(X_L,\cal{L}_L)$ by \cite[Proposition 5.1]{liedtke_2017_good_reduction_k3surfaces}. Define
    \[
        P(X,\cal{L}):=P(X_L,\cal{L}_L)/G_{L/K}.
    \]
    The generic fibre of $P(X,\cal{L})$ is the Galois descent of $X_L$, which is exactly $X$. This shows that $P(X,\cal{L})$ is a projective model of $X$. By the proof of \cite[Proposition 4.7]{liedtke_2017_good_reduction_k3surfaces}, it only depends on $X$ and $\cal{L}$ up to a canonical isomorphism. This proves (1).
    
    Finally, (2) follows since $P(X,\cal{L})_k$ is the Galois descent of $P(X_L,\cal{L}_L)_{k_L}$. Lastly, the minimal desingularization of $P(X,\cal{L})_k$ is the Galois descent of that of $P(X_L,\cal{L}_L)_{k_L}$, therefore a K3 or Enriques surface respectively. It only depends on $X$ up to a canonical isomorphism by the proof of \cite[Proposition 4.7]{liedtke_2017_good_reduction_k3surfaces}. This proves (3).
\end{proof}

\begin{rem}
    In particular, for a K3 or Enriques surface $X$ with good reduction, $X^\circ$ coincides with the special fibre of any smooth model of $X$.
\end{rem}

We are now ready to define the Weyl groups. Let $X$ be a K3 or Enriques surface over $K$ with good reduction after a finite unramified extension. Fix a polarization $\cal{L}$ of $X$. We denote the exceptional locus of $\pi:X^\circ\to P(X,\cal{L})_k$ as $E_{X,\cal{L}}\subset X^\circ$.

Let $\Lambda_{X,\cal{L}}$ be the subgroup of $\Pic(X^\circ)$ generated by the classes of all the integral components of $E_{X,\cal{L}}\subset X^\circ$. For every such component $E_i$, let $k_i$ be the residue field at the singularity $\pi(E_i)$ on $P(X,\cal{L})_k$. Denote $n_i=\dim_{k_i}H^0(E_i,\cal{O}_{E_i})$. Let $s_{E_i}$ be the automorphism of either $\Lambda_{X,\cal{L}}$ or $\Pic(X^\circ)$ defined by
\[
    s_{E_i}:D\mapsto D+\frac{1}{n_i}(D.E_i)E_i.
\]

\begin{defn}
    Let $\Pic(X^\circ)_{\bb{Q}}=\Pic(X^\circ)\otimes\bb{Q}$. The Weyl group $\cal{W}_{X,\cal{L}}$ is the subgroup of $\Aut(\Pic(X^\circ)_{\bb{Q}})$ generated by $s_{E_i}$ for all integral components $E_i$ of $E_{X,\cal{L}}$.
\end{defn}

\begin{rem}
    In \cite[Section 3]{Chiarellotto_2019_neron_ogg_shafarevich_criterion_k3surface}, the Weyl group of a K3 surface is defined as the subgroup of $\Aut(\Pic(X^\circ))$ generated by all $s_{E_i}$. Since $\Pic(X^\circ)$ is torsion-free for a K3 surface, it embeds into $\Pic(X^\circ)_{\bb{Q}}$ and our definition agrees with theirs.

    For an Enriques surface, tensoring with $\bb{Q}$ removes the torsion class $\omega_{X^\circ} \in \Pic(X^\circ)$. We will only be concerned with the torsion-free part of $\Pic(X^\circ)$. Therefore, it is more convenient to work with $\Pic(X^\circ)_{\bb{Q}}$.
\end{rem}

\subsection{Specialization homomorphisms}\label{specialization homomorphism subsection}

Let $X$ be a K3 or Enriques surface over $K$ with good reduction. Let $X^\circ$ be the canonical reduction of $X$.

\begin{lem}\label{birational self map lemma}
    Let $g:X\to X$ be an automorphism. For every smooth model $\cal{X}$ of $X$, $g$ extends uniquely to a birational self-map $g_{\cal{X}}:\cal{X}\dashrightarrow\cal{X}$ defined away from finitely many curves on the special fibre.
\end{lem}
\begin{proof}
    This generalizes \cite[Proposition 2.2(1)]{matsumoto_2021_extendability_automorphisms_k3surfaces} to semi-linear automorphisms. The argument is the same, and we repeat it to keep track of the semi-linearity.

    Let $\sigma_g$ be the automorphism of $K$ induced by $g$. Let $g^*\cal{X}$ be the normalization of $\cal{X}$ in the pullback of $g:X\to X$. Since $\sigma_g$ preserves $\cal{O}_K$, it follows that $g^*\cal{X}$ is another smooth model of $X$ over $\cal{O}_K$. Also, $g:X\to X$ extends to an isomorphism $g_1:g^{*}\cal{X} \to \cal{X}$, semi-linear with respect to $\sigma_g$. By \cite[Proposition 4.7]{liedtke_2017_good_reduction_k3surfaces}, the identity on generic fibres gives a birational map $g_2:\cal{X} \dashrightarrow g^{*}\cal{X}$, which is an isomorphism away from finitely many curves in the special fibre. Then $g_{\cal{X}}=g_1\circ g_2$ gives us the required birational self-map.
\end{proof}

By this lemma, every automorphism $g$ on $X$ extends to $U\subset\cal{X}$, the complement of finitely many curves on the special fibre. Then it restricts to a birational self-map of $X^\circ$. Since $X^\circ$ is a K3 or Enriques surface, the self-map extends uniquely to an automorphism of $X^\circ$.

Let $\Aut(X)$ and $\Aut(X^\circ)$ be the groups of (not necessarily linear) automorphisms of $X$ and $X^\circ$ respectively. Then we obtain a group homomorphism $\mathrm{sp}_{\cal{X}}:\Aut(X)\to\Aut(X^\circ)$ called the specialization homomorphism.

\begin{lem}
    The specialization map $\mathrm{sp}_{\cal{X}}:\Aut(X)\to\Aut(X^\circ)$ is independent of the choice of $\cal{X}$.
\end{lem}
\begin{proof}
    This generalizes \cite[Proposition 2.2(2)]{matsumoto_2021_extendability_automorphisms_k3surfaces} to semi-linear automorphisms with the same proof: let $\cal{X}$ and $\cal{X}'$ be two smooth models of $X$. Let $g\in \Aut(X)$. By the construction in the proof of Lemma \ref{birational self map lemma}, the extension of $g$ to $\cal{X}$ factors as $g_{\cal{X}}:\cal{X}\dashrightarrow g^*\cal{X}\to\cal{X}$. The analogous statement holds for $g_{\cal{X}'}$. Let $f:\cal{X}\dashrightarrow\cal{X}'$, $f':g^*\cal{X}\dashrightarrow g^*\cal{X}'$ be the birational maps induced by the identity on the generic fibres. Then we have the following commutative diagram
    \[
        \xymatrix{
            \cal{X} \ar@{-->}[r] \ar@{-->}[d]_f & g^*\cal{X} \ar[r] \ar@{-->}[d]_{f'} & \cal{X} \ar@{-->}[d]^f \\
            \cal{X}' \ar@{-->}[r] & g^*\cal{X}' \ar[r] & \cal{X}'
        }
    \]
    Note that all the birational maps in this diagram are defined away from finitely many curves on the special fibres by \cite[Proposition 4.7]{liedtke_2017_good_reduction_k3surfaces}. Therefore, the restriction of $g_{\cal{X}}$ and $g_{\cal{X}'}$ to the special fibre $X^\circ$ agree on a dense open subset. This shows that $\mathrm{sp}_{\cal{X}}(g)=\mathrm{sp}_{\cal{X}'}(g)$.
\end{proof}

Given this lemma, we may write $$\mathrm{sp}:\Aut(X)\to\Aut(X^\circ)\ \ \ g\mapsto g_k.$$ Then any group acting on $X$ also acts on $X^\circ$ via specialization. 

In particular, let $X$ be a K3 or Enriques surface over $K$ with good reduction after a finite unramified Galois extension $L/K$. Let $k_L$ be the residue field of $L$. The $G_{L/K}$-action on $X_L$ specializes to a $G_{L/K}$-action on $(X_L)^\circ=X^\circ_{k_L}$. By the proof of Lemma \ref{birational self map lemma}, this factors through the $G_{k_L/k}$-action on $X^\circ_{k_L}$ via the surjection $G_{L/K}\to G_{k_L/k}$.

\subsection{Group cohomology with coefficients in a non-abelian group}\label{1-cocycle sub-section}

We recall the construction of group cohomology with coefficients in a non-abelian group in \cite[Section 5.1, Chapter I]{SerreGaloisCohomology}.

\begin{defn}
    Suppose $G$ is a group acting on a possibly non-abelian group $A$ on the left.
    \begin{enumerate}
        \item An $A$-valued 1-cocycle for $G$ is a function $\phi:G\to A$ such that for every $g,h\in G$,           $$\phi(gh)=\phi(g)g(\phi(h)).$$
        \item Given two 1-cocycles $\phi$ and $\phi'$, we say $\phi\sim\phi'$ if there exists $a\in A$ such that for all $g\in G$,
            $$a\phi(g)=\phi'(g)g(a).$$
        \item Define $H^1(G,A)=\{\text{$A$-valued $1$-cocycles for $G$}\}/\sim$. This is a pointed set, with base point represented by the trivial cocycle $G\to A$, which maps every $g\in G$ to the identity in $A$.
    \end{enumerate}
\end{defn}

Next, suppose $M$ is a left $G$-module, so $G$ acts on the possibly non-abelian group $\Aut(M)$ on the left by conjugation. Denote the action of $G$ on $M$ by $\rho:G\to \Aut(M)$.

\begin{defn}
    Let $c\in H^1(G,\Aut(M))$ be a cohomology class represented by a cocycle $\phi:G\to \Aut(M)$. Define $M^c$, the twist of $M$ by $c$, as a $G$-module:
    \begin{enumerate}
        \item The underlying group is $M$.
        \item The action $\rho^c:G\to \Aut(M^c)$ is defined for every $g\in G$, $m\in M$,
        \[
            \rho^c(g)(m)=\phi(g)(\rho(g)(m)).
        \]
    \end{enumerate}
    This construction is independent of the choice of cocycles up to isomorphisms.
\end{defn}

\begin{lem}\label{twist a module}
    Let $c\in H^1(G,\Aut(M))$ be a cohomology class. Then $c$ is trivial if and only if there exists a $G$-isomorphism 
    \[
        M\cong M^c.
    \]
\end{lem}
\begin{proof}
    Suppose $c$ is represented by a cocycle $\phi:G\to \Aut(M)$. By definition, $c$ is trivial if and only if there exists $f\in \Aut(M)$ such that for every $g\in G$,
    \[
        \phi(g)=f^{-1}\circ g(f)=f^{-1}\circ \rho(g)\circ f\circ \rho(g)^{-1}.
    \]
    Next, a $G$-isomorphism $h:M\to M^c$ is a group automorphism of $M$ such that for every $m\in M$,
    \begin{align*}
        & \rho^c(g)(h(m))=h(\rho(g)(m))\\
        \iff & (\phi(g)\circ\rho(g)\circ h)(m)=(h\circ \rho(g))(m)
    \end{align*}
    Therefore, $\phi(g)=h\circ\rho(g)\circ h^{-1}\circ\rho(g)^{-1}$. Then the two conditions are equivalent upon taking $f=h^{-1}$.
\end{proof}

\section{Extendability and non-abelian group cohomology}

To prove Theorem \ref{main theorem} and Theorem \ref{linear case main theorem}, we define a cohomology class of $G$ that controls the extendability of the $G$-action.

\subsection{Extendability of birational maps}\label{s_f construction subsection} The constructions in this sub-section are based on \cite[Section 7.1]{Chiarellotto_2019_neron_ogg_shafarevich_criterion_k3surface}.

Let $X$ and $Y$ be two K3 or Enriques surfaces over $K$ that are isomorphic to each other. Suppose they have good reduction. Let $\cal{X}$ and $\cal{Y}$ be smooth models of $X$ and $Y$ respectively, with special fibres $X^\circ$ and $Y^\circ$. Let $f:\cal{X} \dashrightarrow\cal{Y}$ be an $\cal{O}_K$-linear birational map inducing a $K$-linear isomorphism $f_K:X \to Y$ on generic fibres. Let $f_k:X^\circ\to Y^\circ$ be the corresponding isomorphism of canonical reductions.

We construct an endomorphism $s_f$ of $\Pic(X^\circ)_{\bb{Q}}$ associated to $f$. Denote by $\Gamma_f\subset \cal{X}\times \cal{Y}$ the graph of $f$ and by $\Gamma_{f,k}\subset X^\circ\times Y^\circ$ its special fibre. Let $p_1$ and $p_2$ be the projections from $X^\circ\times Y^\circ$ to $X^\circ$ and $Y^\circ$ respectively. Then we define a homomorphism
\[
    \Pic(Y^\circ)\to\Pic(X^\circ)\ \ \ D\mapsto p_{1*}(\Gamma_{f,k}\cap p_2^*D),
\]
which extends by linearity to $\tilde{s}_f:\Pic(Y^\circ)_{\bb{Q}}\to\Pic(X^\circ)_{\bb{Q}}$.

We precompose $\tilde{s}_f$ with the pullback via $f_k^{-1}$ and get an endomorphism $s_f$ of $\Pic(X^\circ)_{\bb{Q}}$,
\[
    s_f:\Pic(X^\circ)_{\bb{Q}}\xrightarrow{(f_k^*)^{-1}}\Pic(Y^\circ)_{\bb{Q}}\xrightarrow{\tilde{s}_f}\Pic(X^\circ)_{\bb{Q}}.
\]

Let $\ell\neq p$ be a prime number. The cycle class map defines an injection $[-]:\Pic(X^\circ)_{\bb{Q}}\to H^2_{\et}(X^\circ_{\overline{k}},\bb{Q}_{\ell}(1))$. We extend $s_f$ to $H^2_{\et}(X^\circ_{\overline{k}},\bb{Q}_{\ell}(1))$ as follows. First, we define a homomorphism
\[
    \tilde{s}_{f,\ell}:H^2_{\et}(Y^\circ_{\overline{k}},\bb{Q}_{\ell}(1))\to H^2_{\et}(X^\circ_{\overline{k}},\bb{Q}_{\ell}(1))\ \ \ \alpha\mapsto p_{1*}([\Gamma_{f,k}]\cup p_2^*\alpha).
\]
This is well defined since $[\Gamma_{f,k}]\in H^4_{\et}(X^\circ_{\overline{k}}\times Y^\circ_{\overline{k}},\bb{Q}_\ell(2))$, $p_2^*\alpha\in H^2_{\et}(X^\circ_{\overline{k}}\times Y^\circ_{\overline{k}},\bb{Q}_\ell(1))$, $[\Gamma_{f,k}]\cup p_2^*\alpha\in H^6_{\et}(X^\circ_{\overline{k}}\times Y^\circ_{\overline{k}},\bb{Q}_\ell(3))$, and $p_{1*}([\Gamma_{f,k}]\cup p_2^*\alpha)\in H^2_{\et}(X^\circ_{\overline{k}},\bb{Q}_\ell(1))$.

Similarly, we precompose $\tilde{s}_{f,\ell}$ with the pullback via $f_k^{-1}$ to get an endomorphism $s_{f,\ell}$ of $H^2_{\et}(X^\circ_{\overline{k}},\bb{Q}_\ell(1))$,
\[
    s_{f,\ell}:H^2_{\et}(X^\circ_{\overline{k}},\bb{Q}_\ell(1))\xrightarrow{(f_k^*)^{-1}}H^2_{\et}(Y^\circ_{\overline{k}},\bb{Q}_\ell(1))\xrightarrow{\tilde{s}_{f,\ell}}H^2_{\et}(X^\circ_{\overline{k}},\bb{Q}_\ell(1)).
\]

By \cite[Corollary 10.7]{Etale_cohomology_milne_book}, the cycle class map sends intersection products to cup products and is compatible with pullback. By \cite[Theorem 11.1]{Etale_cohomology_milne_book}, the cycle class map is compatible with the trace map, therefore compatible with pushforward. Hence, the following diagram commutes. 
\begin{equation}\label{l-adic analogue}
    \xymatrix{
        \Pic(X^\circ_{\overline{k}})_{\bb{Q}} 
          \ar[r]^{(f_k^*)^{-1}} \ar[d]_{[-]} &
        \Pic(Y^\circ_{\overline{k}})_{\bb{Q}}  
          \ar[r]^{\tilde{s}_f} \ar[d]_{[-]} &
        \Pic(X^\circ_{\overline{k}})_{\bb{Q}} 
          \ar[d]_{[-]} \\
        H^2_{\et}(X^\circ_{\overline{k}},\bb{Q}_{\ell}(1)) 
          \ar[r]^{(f_k^*)^{-1}} &
        H^2_{\et}(Y^\circ_{\overline{k}},\bb{Q}_{\ell}(1)) 
          \ar[r]^{\tilde{s}_{f,\ell}} &
        H^2_{\et}(X^\circ_{\overline{k}},\bb{Q}_{\ell}(1))}
\end{equation}

Also, we have the following characterization of $\tilde{s}_{f,\ell}$.

\begin{lem}[{\cite[Lemma 5.6]{liedtke_2017_good_reduction_k3surfaces}}]\label{generic pullback}
    Let $\phi$ be the comparison isomorphism provided by the smooth base change theorem. Then the following diagram commutes.
    \[
        \xymatrix{
            H^2_{\et}(Y^\circ_{\overline{k}},\bb{Q}_{\ell}(1)) \ar@{->}[r]^\phi \ar@{->}[d]_{\tilde{s}_{f,\ell}} & H^2_{\et}(Y_{K^{\mathrm{s}}},\bb{Q}_{\ell}(1)) \ar@{->}[d]_{f_K^*}\\ H^2_{\et}(X^\circ_{\overline{k}},\bb{Q}_{\ell}(1)) \ar@{->}[r]^\phi & H^2_{\et}(X_{K^{\mathrm{s}}},\bb{Q}_{\ell}(1))}
    \]
Therefore, $s_{f,\ell}$ decomposes as follows 
    \begin{align*}
        s_{f,\ell}: & H^2_{\et}(X^\circ_{\overline{k}},\bb{Q}_{\ell}(1)) \xrightarrow{(f_k^*)^{-1}} H^2_{\et}(Y^\circ_{\overline{k}},\bb{Q}_{\ell}(1))\xrightarrow{\phi} H^2_{\et}(Y_{K^{\mathrm{s}}},\bb{Q}_{\ell}(1)) \\ \xrightarrow{f^*_K} & H^2_{\et}(X_{K^{\mathrm{s}}},\bb{Q}_{\ell}(1))\xrightarrow{\phi^{-1}} H^2_{\et}(X^\circ_{\overline{k}},\bb{Q}_{\ell}(1))
    \end{align*}
\end{lem}

\begin{prop}[cf. {\cite[Proposition 7.4]{Chiarellotto_2019_neron_ogg_shafarevich_criterion_k3surface}}]\label{properties of s_f}
    The constructions of $s_f$ and $s_{f,\ell}$ have the following properties:
    \begin{enumerate}
        \item For every $\alpha_1,\alpha_2\in H^2_{\et}(X^\circ_{\overline{k}},\bb{Q}_{\ell}(1))$, $$s_{f,\ell}(\alpha_1)\cup s_{f,\ell}(\alpha_2)=\alpha_1\cup\alpha_2.$$ In particular, for every $D_1,D_2\in \Pic(X^\circ)_{\bb{Q}}$, $$s_f(D_1).s_f(D_2)=D_1.D_2.$$
        \item For every two composable birational maps $f$ and $g$, $$s_{g\circ f,\ell}=s_{f,\ell}\circ f_k^*\circ s_{g,\ell}\circ (f_k^*)^{-1},$$ $$s_{g\circ f}=s_f\circ f_k^*\circ s_g\circ (f_k^*)^{-1}.$$
        \item For every birational map $f$, $s_{f,\ell}$ and $s_f$ are invertible with $$(s_{f,\ell})^{-1}=f_k^*\circ s_{f^{-1},\ell}\circ (f_k^*)^{-1},$$
        $$(s_f)^{-1}=f_k^*\circ s_{f^{-1}}\circ (f_k^*)^{-1}.$$
    \end{enumerate}
\end{prop}
{\begin{proof}
    This proposition generalizes \cite[Proposition 7.4]{Chiarellotto_2019_neron_ogg_shafarevich_criterion_k3surface} to Enriques surfaces essentially with the same proof if we replace $\Pic(X^\circ)$ in \cite[Proposition 7.4]{Chiarellotto_2019_neron_ogg_shafarevich_criterion_k3surface} by $\Pic(X^\circ)_{\bb{Q}} $.

    In the commutative diagram \ref{l-adic analogue}, the vertical maps are injective. Therefore, it suffices to prove these for $s_{f,\ell}$. Also, $(2)\implies(3)$, so we are left with (1) and (2). By Lemma \ref{generic pullback}, we only need to prove that pullback, pushforward, and the comparison isomorphism are compatible with cup products and composition. This is clear.
\end{proof}

Let $\cal{L}$ be a polarization of $X$. Let $f:\cal{X}\dashrightarrow\cal{Y}$ be an $\cal{O}_K$-linear birational map between smooth models of $X$ and $Y$ that restricts to a $K$-linear isomorphism on the generic fibre. We are interested in when $s_f$ lies in the Weyl group $\cal{W}_{X,\cal{L}}$. This leads to the following definition.

\begin{defn}
    Suppose $\cal{X}$ is $\cal{L}$-terminal. Denote the pushforward of $\cal{L}$ via $f_K:X\to Y$ by $f_*\cal{L}$. Then we say $f$ is $\cal{L}$-terminal if $\cal{Y}$ is $f_*\cal{L}$-terminal.
\end{defn}

\begin{prop}[cf. {\cite[Proposition 7.5]{Chiarellotto_2019_neron_ogg_shafarevich_criterion_k3surface}}]\label{surjectivity}
    Suppose $\cal{X}$ is $\cal{L}$-terminal. Assume all irreducible components of $E_{X,\cal{L}}$ are geometrically irreducible. Then for any element $w\in \cal{W}_{X,\cal{L}}$, there exists
    a smooth model $\cal{X}^+$ of $X$, such that the birational map $f:\cal{X}\dashrightarrow\cal{X}^+$ induced by identity on $X$ is $\cal{L}$-terminal and $s_f=w$.
\end{prop}
\begin{proof}
    This proposition generalizes \cite[Proposition 7.5]{Chiarellotto_2019_neron_ogg_shafarevich_criterion_k3surface} to Enriques surfaces, essentially with the same proof. The only difference is that we need to replace $\Pic(X^\circ)$ by $\Pic(X^\circ)_{\bb{Q}} $. We give a sketch of the proof.
    
    By Proposition \ref{properties of s_f}(2), it suffices to prove this for $w=s_E$, for an irreducible component $E\subset E_{X,\cal{L}}$. Since $E$ is a $(-2)$-curve, there exists a blow down $\cal{X}\to\cal{X}'$ of $E$, which contracts $E$ to $y\in \cal{X'}
    $ by \cite[Theorem 3.9]{algebraic_surface}. Let $\widehat{\cal{X}_y'}=\Spf(R)$ be the formal completion of $\cal{X}'$ at $y$ and let $\widehat{\cal{X}}\to\widehat{\cal{X}_y'}$ be the formal fibre. Following the proof of \cite[Proposition 4.2]{liedtke_2017_good_reduction_k3surfaces}, we find an automorphism $t:\widehat{\cal{X}'_y}\to \widehat{\cal{X}'_y}$ which induces $-1$ on the class group $\Cl(R)$. By \cite[Theorem 3.2]{artin_dilatation}, there exists a dilatation of algebraic spaces $f^+:\cal{X}^+\to\cal{X}'$, such that the formal completion of $f^+$ along its exceptional locus coincides with $\widehat{\cal{X}}\to\widehat{\cal{X}_y'}\xrightarrow{t}\widehat{\cal{X}_y'}$. This defines a birational map $$f:\cal{X}\to\cal{X}'\overset{t^{-1}}{\dashrightarrow}\cal{X}^+.$$ We claim that $s_f=s_E$.

    Let $\Gamma_{f,k}$ be the special fibre of the graph of $f$ and let $\Gamma_{f_k}$ be the graph of $f_k$. Since $f$ is an isomorphism away from $E$, we see that $$\Gamma_{f,k}=\Gamma_{f_k}+b(E\times f_k(E))$$ for some integer $b\geq0$. Then $$s_f(D)=D+b(E.D)E.$$ Also, by Proposition \ref{properties of s_f}(1), $s_f$ preserves the intersection pairing, thus
    \begin{align*}
        & E^2 = s_f(E)^2 \\
        \implies & E^2 = E^2 + 2b(E^2)^2 + b^2(E^2)^3 \\
        \implies & 2b(E^2)^2 + b^2(E^2)^3 = 0.
    \end{align*}
    Setting $E^2=-2$ yields $b=0$ or $1$. If $b=0$, $f$ extends to an isomorphism. If $b=1$, $s_f=s_E$. Then we prove $s_f=s_E$ by arguing that $f$ does not extend to an isomorphism as in \cite[Proposition 7.5]{Chiarellotto_2019_neron_ogg_shafarevich_criterion_k3surface}.

    Lastly, we need to prove that $f$ is $\cal{L}$-terminal. Since $\cal{X}$ is $\cal{L}$-terminal, $\cal{L}_k.E=0$. We may choose a divisor corresponding to $\cal{L}$ that does not intersect $E$. Since $f$ is an isomorphism away from $E$, it follows that $(f_*\cal{L})_k$ is also big and nef.
\end{proof}

Conversely, we have the following.

\begin{prop}[cf. {\cite[Theorem 7.6]{Chiarellotto_2019_neron_ogg_shafarevich_criterion_k3surface}}]\label{birational to weyl group}
Let $f:\cal{X}\dashrightarrow\cal{Y}$ be an $\cal{L}$-terminal birational map. Then
\begin{enumerate}
    \item $s_f\in\cal{W}_{X,\cal{L}}$.
    \item $f$ is an isomorphism if and only if $s_f=\mathrm{id}$.
\end{enumerate}
\end{prop}
\begin{proof}
    Again, this proposition generalizes \cite[Theorem 7.6]{Chiarellotto_2019_neron_ogg_shafarevich_criterion_k3surface} to Enriques surfaces with the same proof. We give a sketch of the proof.

    For a finite and unramified Galois extension $L/K$, we have $\cal{W}_{X,\cal{L}}=\cal{W}_{X_L,\cal{L}_L}^{G_{L/K}}$ by \cite[Corollary 3.3]{Chiarellotto_2019_neron_ogg_shafarevich_criterion_k3surface}. Also, whether $f$ extends to an isomorphism can be detected over $L$. Therefore, we extend $K$ so that the irreducible components $E_i\subset E_{X,\cal{L}}$ and $E_j'\subset E_{Y,f_*\cal{L}}$ are geometrically irreducible and the singularities of $P(X,\cal{L})_k$ are of type ADE.

    Since $f$ restricts to an isomorphism $f_K:X\to Y$, it induces an isomorphism $P(X,\cal{L})\to P(Y,f_*\cal{L})$. Next, $X^\circ$ and $Y^\circ$ are respectively the minimal desingularizations of $P(X,\cal{L})_k$ and $P(Y,f_*\cal{L})_k$ with exceptional divisors $E_{X,\cal{L}}$ and $E_{Y,f_*\cal{L}}$. Thus, $f$ induces an isomorphism $X^\circ\setminus E_{X,\cal{L}}\cong Y^\circ\setminus E_{Y,f_*\cal{L}}$. Let $p_1$ and $p_2$ respectively be the projections from $X^\circ\times Y^\circ$ to the two factors. Let $\Gamma_f\subset \cal{X}\times\cal{Y}$ be the graph of $f$ with special fibre $\Gamma_{f,k}$. Let $\Gamma_{f_k}\subset X^\circ\times Y^\circ$ be the graph of $f_k$. Then we have $s_f:D\mapsto p_{1*}(\Gamma_{f,k}\cap p_2^*D)$ with
    \[
        \Gamma_{f,k}=\Gamma_{f_k}+\sum_{i,j}b_{ij}(E_i\times E_j')
    \]
    for integers $b_{ij}\geq 0$. Therefore,
    \[
        s_f(D)=D+\sum_{i,j}b_{ij}((f_k^*)^{-1}D.E_j')E_i.
    \]
    In particular, $s_f\in\Aut(\Pic(X^\circ)_{\bb{Q}})$ preserves the integral lattice $\Lambda_{X,\cal{L}}\subset\Pic(X^\circ)_{\bb{Q}}$. Also, the isomorphism $X^\circ\to Y^\circ$ maps the connected components of $E_{X,\cal{L}}$ to those of $E_{Y,f_*\cal{L}}$, so $b_{ij}=0$ if $E_i$ and $f_k^{-1}(E_j')$ are in different connected components. Then $s_f|_{\Lambda_{X,\cal{L}}}$ preserves every sub-lattice $\Lambda_l\subset \Lambda_{X,\cal{L}}$ corresponding to each connected component of $E_{X,\cal{L}}$. By Proposition \ref{properties of s_f}(3), the same thing holds for $s_{f^{-1}}$.

    Moreover, by Proposition \ref{properties of s_f}(1), $s_f$ preserves intersection pairing, so we write $$s_f|_{\Lambda_{X,\cal{L}}}=\prod_{l}\tilde{w}_l$$ for $\tilde{w}_l\in O(\Lambda_l)$ in the orthogonal group of every sub-lattice $\Lambda_l$. By \cite[Theorem 1]{automorphism_of_lattices}, for every $l$, $\tilde{w}_l=w_l\alpha_l$ where $w_l$ is in the Weyl group of $\Lambda_l$ and $\alpha_l$ induced by an automorphism of the Dynkin diagram of $\Lambda_l$. By Proposition \ref{surjectivity}, we may absorb $\prod_lw_l$ into $s_f$ and assume $$s_f|_{\Lambda_{X,\cal{L}}}=\prod_l\alpha_l.$$
    To prove (1), it suffices to prove that $s_f=\mathrm{id}$.
    Let $A$ be the intersection matrix $A$ of $E_{X,\cal{L}}$ and let $B=\{b_{ij}\}$. By \cite[Lemma 3.5]{Chiarellotto_2019_neron_ogg_shafarevich_criterion_k3surface}, $s_f|_{\Lambda_{X,\cal{L}}}=\mathrm{id}$. Therefore, $BA=0$. Since $A$ is negative definite, $B=0$ and $s_f=\mathrm{id}.$

    The proof of (2) is similar. If $f$ is an isomorphism, $s_f=\mathrm{id}$. Conversely, $s_f=\mathrm{id}$ implies that $BA=0$. Since $A$ is negative definite, $B=0$ and $f$ is an isomorphism.
\end{proof}

Given Proposition \ref{birational to weyl group}, we can remove the assumption in Proposition \ref{surjectivity} that all irreducible components of $E_{X,\cal{L}}$ are geometrically irreducible.

\begin{cor}\label{surjectivity improved}
    Let $\cal{X}$ be an $\cal{L}$-terminal model of $X$. For any $w\in\cal{W}_{X,\cal{L}}$, there exists
    a smooth model $\cal{X}^+$ of $X$, such that the birational map $f:\cal{X}\dashrightarrow\cal{X}^+$ induced by identity on $X$ is $\cal{L}$-terminal and $s_f=w$.
\end{cor}

\begin{proof}
    This is an application of Proposition \ref{surjectivity} and Proposition \ref{birational to weyl group}. The proof is the same as \cite[Corollary 7.7]{Chiarellotto_2019_neron_ogg_shafarevich_criterion_k3surface}.
\end{proof}

\subsection{Extendability of group actions}

The constructions in this sub-section are based on \cite[Section 7.2]{Chiarellotto_2019_neron_ogg_shafarevich_criterion_k3surface}. We want to apply the results from the previous sub-section to deduce a criterion for the extendability of group actions. We begin with some terminology.

Let $G$ be a group acting on the right on a proper variety $X$ (finite type, geometrically integral, and separated) over a field $F$.

For every $g\in G$, $g$ induces an automorphism of $F$ by pulling back the global sections. This defines a homomorphism $\sigma:G\to \Aut(F)$. Denote by $\sigma_g$ the automorphism induced by $g$ on $\Spec F$.

\begin{defn}\label{base field}
The base field of the $G$-action is the fixed field of $\sigma(G)\subset\Aut(F)$.
\end{defn}

Let $g\in G$. Then $g$ is an automorphism of $X$ semi-linear with respect to $\sigma_g$. Equivalently, denote by $X^{\sigma_g}$ the base change of $X$ by $\sigma_g$. Then the induced isomorphism $g^{\#}:X\to X^{\sigma_g}$ is $F$-linear. In short, we have the following commutative diagram, where the square is Cartesian. $$\xymatrix{X \ar@{->}[r]^{g^{\#}} \ar@{->}[dr]_g & X^{\sigma_g} \ar@{->}[r] \ar@{->}[d] & \Spec F\ar@{->}[d]^{\sigma_g} \\  & X \ar@{->}[r] & \Spec F} $$

Now we return to the case where $X$ is a K3 or Enriques surface over $K$ with good reduction. Let $\cal{X}$ be a smooth model of $X$. Let $G$ be a finite group acting on $X$ on the right with base field $K_0$.

\begin{assump}
    We assume that $K/K_0$ is finite and unramified.
\end{assump}

Let $g\in G$. As above, we have a factorization $g:X\xrightarrow{g^{\#}}X^{\sigma_g}\to X$. The second isomorphism is induced by $\sigma_g$. This obviously extends to an isomorphism $\cal{X}^{\sigma_g}\to\cal{X}$, where we abuse notation and write $\sigma_g$ also for the induced automorphism of $\Spec\cal{O}_K$. Therefore, the birational map $g_{\cal{X}}:\cal{X}\dashrightarrow\cal{X}$ induced by $g$ extends to an automorphism if and only if the birational map $g^{\#}_{\cal{X}}:\cal{X}\dashrightarrow\cal{X}^{\sigma_g}$ induced by $g^\#$ extends to an isomorphism.

To apply results from Section \ref{s_f construction subsection}, we want to find some polarization $\cal{L}$ of $X$ and an $\cal{L}$-terminal model $\cal{X}$, such that $s_{g^{\#}_\cal{X}}\in \cal{W}_{X,\cal{L}}$ for every $g\in G$.

\begin{lem}
    Suppose $\cal{L}$ is a polarization of $X$ fixed by the $G$-action. Let $\cal{X}$ be an $\cal{L}$-terminal model. Then for every $g\in G$, $s_{g^{\#}_\cal{X}}\in \cal{W}_{X,\cal{L}}$.
\end{lem}
\begin{proof}
    Let $\cal{L}^{\sigma_g}$ be the polarization on $\cal{X}^{\sigma_g}$ obtained as the base change of $\cal{L}$ by $\sigma_g$. Then $X^{\sigma_g}$ is $\cal{L}^{\sigma_g}$-terminal. Next, since $\cal{L}$ is fixed by $G$, it follows that $(g^\#)^*\cal{L}^{\sigma_g}\cong \cal{L}$. This shows that $g^\#_\cal{X}$ is $\cal{L}$-terminal. Then we are done by Proposition \ref{birational to weyl group}.
\end{proof}

\begin{rem}
    Since $G$ is finite, the $G$-action fixes some polarization. Indeed, given any polarization $\mathcal{L}$, the tensor product $\bigotimes_{g \in G} g^*\mathcal{L}$ is $G$-invariant and hence defines a polarization fixed by $G$.
\end{rem}

Recall that, via the specialization homomorphism, $G$ also acts on $X^\circ$ on the right. Therefore, $G$ acts on $\Pic(X^\circ)$ on the left via pullback, and on $\Aut(\Pic(X^\circ))$ on the left via conjugation.

Equivalently, the $G$-action on $\Aut(\Pic(X^\circ))$ can be described as follows. We abuse notation and write $\sigma_g$ both for the automorphisms of $\Spec K$ and $\Spec k$ induced respectively by $g\in \Aut(X)$ and $g_k\in \Aut(X^\circ)$. Given $s\in \Aut(\Pic(X^\circ)_{\bb{Q}})$, let $s^{\sigma_g}\in \Aut(\Pic((X^\circ)^{\sigma_g})_{\bb{Q}})$ be the base change of $s$ by $\sigma_g\in\Aut(\Spec k)$, i.e. $s^{\sigma_g}$ is the unique automorphism of $\Pic((X^\circ)^{\sigma_g})_{\bb{Q}}$ that makes the following diagram commute.
$$\xymatrix{\Pic(X^\circ)_{\bb{Q}} \ar@{->}[r]^{s} \ar@{->}[d]^{\sigma_g^*} & \Pic(X^\circ)_{\bb{Q}} \ar@{->}[d]^{\sigma_g^*} \\ \Pic((X^\circ)^{\sigma_g})_{\bb{Q}} \ar@{->}[r]^{s^{\sigma_g}} & \Pic((X^\circ)^{\sigma_g})_{\bb{Q}}}$$
We claim that $g$ acts on $\Aut(\Pic(X^\circ)_{\bb{Q}})$ equivalently via $$s\mapsto (g_k^{\#})^*\circ s^{\sigma_g}\circ ((g_k^{\#})^*)^{-1}.$$
To see this, we need to prove
\begin{equation}\label{picard group action}
    g_k^*\circ s\circ (g_k^*)^{-1}=(g_k^{\#})^*\circ s^{\sigma_g}\circ ((g_k^{\#})^*)^{-1}.
\end{equation}
Since $g_k=\sigma_g\circ g^\#$, we have $g_k^*=(g^\#)^*\circ\sigma_g^*$.
Then the LHS and RHS of \ref{picard group action} fit in the following diagram. $$\xymatrix{\Pic(X^\circ)_{\bb{Q}} \ar@{->}[r]^{((g^{\#})^*)^{-1}} & \Pic((X^\circ)^{\sigma_g})_{\bb{Q}} \ar@{->}[r]^{(\sigma_g^*)^{-1}} \ar@{->}[d]_{s^{\sigma_g}} & \Pic(X^\circ)_{\bb{Q}} \ar@{->}[d]_{s} \\ \Pic(X^\circ)_{\bb{Q}} & \Pic((X^\circ)^{\sigma_g})_{\bb{Q}} \ar@{->}[l]_{(g^{\#})^*} & \Pic(X^\circ)_{\bb{Q}} \ar@{->}[l]_{\sigma_g^*}}$$
To prove \ref{picard group action}, it amounts to showing the square in the diagram commutes, which is immediate by the construction of $s^{\sigma_g}$.

Let $\cal{L}$ be a polarization of $X$ fixed by the $G$-action and let $g\in G$. Denote by $\cal{L}^{\sigma_g}$ the base change of $\cal{L}$ by $\sigma_g$. Then $g^{\#}$ induces an isomorphism $P(X,\cal{L})\to P(X^{\sigma_g},\cal{L}^{\sigma_g})$, therefore an isomorphism of their desingularizations $X^\circ\to (X^\circ)^{\sigma_g}$. In particular, $g^{\#}$ maps the exceptional divisors of $P(X,\cal{L})$ to those of $P(X^{\sigma_g},\cal{L}^{\sigma_g})$. This implies that the $G$-action on $\Aut(\Pic(X^\circ)_{\bb{Q}})$ restricts to a $G$-action on $\cal{W}_{X,\cal{L}}$.

\begin{lem}\label{ugly cocycle lemma}
    For any $\cal{L}$-terminal model $\cal{X}$, the function $$\alpha_{\cal{L},\cal{X}}:G\to \cal{W}_{X,\cal{L}}\ \ \ g\mapsto s_{g^{\#}_\cal{X}}$$
    defines a $1$-cocycle of $G$ with values in $\cal{W}_{X,\cal{L}}$.
\end{lem}
\begin{proof}
    Let $g,h\in G$. We abuse notation and write $\sigma_g$ also for the induced base change isomorphism $X^{\sigma_g}\to X$. Let $(h^{\#})^{\sigma_g}=\sigma_g^{-1}\circ h^\#\circ\sigma_g$ be the base change of $h^{\#}$ by $\sigma_g$, i.e. the following diagram commutes.
    $$\xymatrix{X^{\sigma_g} \ar@{->}[r]^{(h^{\#})^{\sigma_g}} \ar@{->}[d]^{\sigma_g} & {(X^{\sigma_h})^{\sigma_g}} \ar@{->}[d]^{\sigma_g} \\ X \ar@{->}[r]^{h^\#} & X^{\sigma_h}}$$
    We use the subscript $\cal{X}$ to denote the birational map induced on $\cal{X}$. The construction of $s_{h^\#_\cal{X}}$ is compatible with base change, so we have $s_{(h^{\#}_{\cal{X}})^{\sigma_g}}=s_{(h^{\#}_{\cal{X}})}^{\sigma_g}$.
    
    Also, by construction, we have $\sigma_{h\circ g}=\sigma_h\circ\sigma_g$. Therefore,
    \begin{align*}
        (h^{\#})^{\sigma_g}\circ g^{\#} & = \sigma_g^{-1}\circ h^\#\circ\sigma_g \circ g^\# \\ & = \sigma_g^{-1}\circ(\sigma_h^{-1}\circ h)\circ\sigma_g\circ(\sigma_g^{-1}\circ g) \\ & = \sigma_{h\circ g}^{-1}\circ(h\circ g) = (h\circ g)^\#.
    \end{align*}
    In particular, we have $(h^{\#}_\cal{X})^{\sigma_g}\circ g_{\cal{X}}^{\#}=(h\circ g)_{\cal{X}}^\#$.
    
    Then we compute
    \begin{align*}
        \alpha_{\cal{L},\cal{X}}(gh) & = s_{(hg)^{\#}_\cal{X}}=s_{(h^{\#}_{\cal{X}})^{\sigma_g}\circ g^{\#}_\cal{X}} \\ & = s_{g^{\#}_\cal{X}}\circ (g_k^{\#})^*\circ s_{(h^{\#}_{\cal{X}})^{\sigma_g}}\circ ((g_k^{\#})^*)^{-1} \\ & = s_{g^{\#}_\cal{X}}\circ (g_k^{\#})^*\circ s_{h^{\#}_{\cal{X}}}^{\sigma_g}\circ ((g_k^{\#})^*)^{-1} \\ & = \alpha_{\cal{L},\cal{X}}(g)\circ g(\alpha_{\cal{L},\cal{X}}(h)).
    \end{align*}
    This proves that $\alpha_{\cal{L},\cal{X}}$ is a 1-cocycle.
\end{proof}

Using Proposition \ref{birational to weyl group}, we obtain the next lemma which says this cocycle controls the extendability of the $G$-action to $\cal{X}$:

\begin{lem}\label{extendability to a model}
    The $G$-action on $X$ extends to $\cal{X}$ if and only if the $1$-cocycle $\alpha_{\cal{L},\cal{X}}$ is trivial, i.e. $\alpha_{\cal{L},\cal{X}}(g)=1$ for every $g\in G$.
\end{lem}

Another natural question is how $\alpha_{\cal{L},\cal{X}}$ changes as we choose different $\cal{X}$.
\begin{lem}\label{cohomology class and the choice of models}
    Let $\cal{X}, \cal{X}^+$ be two $\cal{L}$-terminal models and $f:\cal{X}\dashrightarrow\cal{X}^+$ be the birational map induced by identity on $X$. Then for every $g\in G$, $$\alpha_{\cal{L},\cal{X}^+}(g)=s_f^{-1}\circ\alpha_{\cal{L},\cal{X}}(g)\circ g(s_f).$$ In particular, the cohomology class represented by $\alpha_{\cal{L},\cal{X}}$ is independent of the choice of $\cal{L}$-terminal model $\cal{X}$. We denote this class by $\alpha_{\cal{L}}$.
\end{lem}
\begin{proof}
    View $\alpha_{\cal{L},\cal{X}^+}$ as a function
    $$\alpha_{\cal{L},\cal{X}^+}:G\to \Aut(\Pic(\cal{X}^+_k)_{\bb{Q}}),$$ i.e. we don't identify $X^\circ=\cal{X}_k$ with $\cal{X}^+_k$.
    Then we want to prove
    $$f_k^*\circ\alpha_{\cal{L},\cal{X}^+}(g)\circ (f^*_k)^{-1}=s_f^{-1}\circ\alpha_{\cal{L},\cal{X}}(g)\circ g(s_f).$$
    Denote by $g_{\cal{X}^+}$ the birational self-map on $\cal{X}^+$ induced by $g$ and by $f^{\sigma_g}$ the base change of $f$ by $\sigma_g$. Then the following diagram commutes.
    $$\xymatrix{ \cal{X} \ar@{-->}[r]^-{g^{\#}_\cal{X}} \ar@{-->}[d]_f & \cal{X}^{\sigma_g} \ar@{-->}[d]^{f^{\sigma_g}} \\ \cal{X}^+ \ar@{-->}[r]^-{g_{\cal{X}^+}^{\#}} & (\cal{X}^+)^{\sigma_g}}$$
    In particular, $s_{f^{\sigma_g}\circ g^{\#}_\cal{X}} = s_{g^{\#}_{\cal{X}^+}\circ f}$.
    We compute
    \begin{align*}
        \alpha_{\cal{L},\cal{X}}(g)\circ g(s_f) & = s_{g^{\#}_{\cal{X}}}\circ (g_k^{\#})^* \circ s_{f^{\sigma_g}} \circ ((g_k^{\#})^*)^{-1} \\ & = s_{f^{\sigma_g}\circ g^{\#}_\cal{X}} = s_{g^{\#}_{\cal{X}^+}\circ f} \\ & = s_f\circ f^*_{k}\circ \alpha_{\cal{L},\cal{X}^+}(g)\circ (f_k^*)^{-1}
    \end{align*}
    This finishes the proof.
\end{proof}

This gives a criterion for the extendability of group actions.

\begin{thm}\label{first criterion}
    Let $X$ be a K3 or Enriques surface over $K$ with good reduction. Let $G$ be a finite group acting on $X$ with base field $K_0$, such that $K/K_0$ is finite and unramified. Then the following are equivalent.
    \begin{enumerate}
        \item The action of $G$ is extendable.
        \item The cohomology class $\alpha_{\cal{L}}$ is trivial for some polarization $\cal{L}$ fixed by $G$.
        \item The cohomology class $\alpha_{\cal{L}}$ is trivial for every polarization $\cal{L}$ fixed by $G$.
    \end{enumerate}
\end{thm}
\begin{proof}
    Since $(3)\implies (2)$, it suffices to prove that $(1)\implies (3)$ and $(2)\implies (1)$. Assume the action of $G$ is extendable and let $\cal{L}$ be a polarization of $X$ fixed by $G$. Then the action extends to some $\cal{L}$-terminal model by \cite[Proposition 4.6]{matsumoto_2021_extendability_automorphisms_k3surfaces}, which holds for group actions that are not necessarily linear without any change to the proof. Then the cohomology class $\alpha_\cal{L}$ is trivial by Lemma \ref{extendability to a model} and Lemma \ref{cohomology class and the choice of models}. This shows $(1)\implies (3)$.

    Next, we prove $(2)\implies (1)$. Suppose the cohomology class $\alpha_{\cal{L}}$ is trivial for some $\cal{L}$ fixed by $G$. Let $\cal{X}$ be an $\cal{L}$-terminal model of $X$. Then there exists some $w\in \cal{W}_{X,\cal{L}}$ such that
    \[
        \alpha_{\cal{L},\cal{X}}(g)=w^{-1}g(w).
    \]
    By Corollary \ref{surjectivity improved}, there exists a smooth model $\cal{X}^+$ of $X$ such that the birational map $f:\cal{X}\dashrightarrow\cal{X}^+$ induced by identity on $X$ is $\cal{L}$-terminal and $w^{-1}=s_f$. Then by Lemma \ref{cohomology class and the choice of models}, for every $g\in G$,
    \[
        \alpha_{\cal{L},\cal{X}^+}(g)=s_f^{-1}\circ\alpha_{\cal{L},\cal{X}}(g)\circ g(s_f)=\mathrm{id}.
    \]
    This proves that the action of $G$ extends to $\cal{X}^+$ by Lemma \ref{extendability to a model}.
\end{proof}

\section{\texorpdfstring{Extendability and $\ell$-adic cohomology}{Extendability and l-adic cohomology}}\label{proof of the main results}

This section adapts the methods of \cite[Section 8 and 9]{Chiarellotto_2019_neron_ogg_shafarevich_criterion_k3surface} to finite group actions (not necessarily linear) and also to Enriques surfaces. We will first define the $\ell$-adic realization of $\alpha_\cal{L}$ and prove Theorem \ref{main theorem} and Theorem \ref{linear case main theorem}. Let $X$ be a K3 or Enriques surface over $K$ with good reduction. Let $G$ be a finite group acting on $X$ with base field $K_0$, such that $K/K_0$ is finite and unramified. Denote by $\sigma:G\to G_{K/K_0}$ the action of $G$ on $K$. Let $\cal{L}$ be a polarization on $X$ fixed by the $G$-action. Recall that we denote the subgroup of $\Pic(X^\circ)$ generated by the exceptional divisors of the minimal desingularization $X^\circ\to P(X,\cal{L})_k$ by $\Lambda_{X,\cal{L}}$. Let $\Lambda_{X,\cal{L},\bb{Q}_\ell}=\Lambda_{X,\cal{L}}\otimes\bb{Q}_{\ell}$.

\subsection{$G$-action and $\ell$-adic cohomology}\label{fibre product action}

Note that since $G$ does not necessarily act $K$-linearly on $X$, there is no natural $G$-action on $X_{K^{\mathrm{s}}}$. Instead, we consider the fibre product $\overline{G}=G\times_{G_{K/K_0}}G_{K_0}$ of $\sigma:G\to G_{K/K_0}$ and the restriction homomorphism $G_{K_0}\to G_{K/K_0}$. We define a right $\overline{G}$-action on $X_{K^\mathrm{s}}$ by letting $G$ act on $X$ and $G_{K_0}$ act on $\Spec K^{\mathrm{s}}$. Therefore, $\overline{G}$ acts on $H^2_{\et}(X_{K^{\mathrm{s}}},\bb{Q}_\ell(1))$ on the left via pullback, and on $\Aut(H^2_{\et}(X_{K^{\mathrm{s}}},\bb{Q}_\ell(1)))$ on the left via conjugation.

Via specialization, $G$ also acts on $X^\circ$. Let $k_0$ be the residue field of $K_0$.  We form the fibre product $\overline{G}_k=G\times_{G_{k/k_0}}G_{k_0}$. Similarly, $\overline{G}_k$ acts on $X^\circ_{\overline{k}}$, $H^2_{\et}(X^\circ_{\overline{k}},\bb{Q}_\ell(1))$ and $\Aut(H^2_{\et}(X^\circ_{\overline{k}},\bb{Q}_\ell(1)))$.

Moreover, the natural surjection $G_{K_0}\to G_{k_0}$ induces a surjection $\overline{G}\to\overline{G}_k$ with kernel $I_K$. Since $X$ has good reduction, $I_K$ acts trivially on $H^2_{\et}(X_{K^{\mathrm{s}}},\bb{Q}_\ell(1))$. Therefore, $\overline{G}_k\cong\overline{G}/I_K$ acts on $H^2_{\et}(X_{K^{\mathrm{s}}},\bb{Q}_\ell(1))$.

Note that we have the following short exact sequence \begin{equation}\label{semi-linear short exact sequence}
    1\to G_{k}\to\overline{G}_k\to G\to 1
\end{equation}
where the first homomorphism maps $\sigma$ to $(\mathrm{id}, \sigma)$ and the second homomorphism is projection. Then $G$ acts on $\Aut_{G_k}(H^2_{\et}(X_{K^{\mathrm{s}}},\bb{Q}_{\ell}(1)))$, $\Aut_{G_k}(H^2_{\et}(X^\circ_{\overline{k}},\bb{Q}_{\ell}(1)))$, $\Aut(H^2_{\et}(X_{K^{\mathrm{s}}},\bb{Q}_\ell(1))^{G_k})$, and $\Aut(H^2_{\et}(X^\circ_{\overline{k}},\bb{Q}_\ell(1))^{G_k})$. 

\subsection{$\ell$-adic realization of $\alpha_{\cal{L}}$}

Recall that for every irreducible component $E_i\subset E_{X,\cal{L}}$, we denote by $k_i$ the residue field of the singularity on $P(X,\cal{L})_k$ that $E_i$ corresponds to, and $n_i=\dim_{k_i}H^0(E_i,\cal{O}_{E_i})$.

\begin{lem}
    There exists a $G$-equivariant group homomorphism 
    \[
        i_{\ell}:\cal{W}_{X,\cal{L}}\to \Aut_{G_k}(H^2_{\et}(X^\circ_{\overline{k}},\bb{Q}_{\ell}(1)))
    \]
    such that for every irreducible component $E_i\subset E_{X,\cal{L}}$, $s_{E_i}$ maps to the automorphism 
    \[
        H^2_{\et}(X^\circ_{\overline{k}},\bb{Q}_{\ell}(1))\to H^2_{\et}(X^\circ_{\overline{k}},\bb{Q}_{\ell}(1))\ \ \ \alpha\mapsto \alpha+\frac{1}{n_i}([E_i]\cup\alpha)[E_i].
    \]
\end{lem}
\begin{proof}
    We construct such $i_\ell$ as follows. There is a natural action of $\cal{W}_{X,\cal{L}}$ on $\Lambda_{X,\cal{L},\bb{Q}_\ell}$, which embeds into $H^2_{\et}(X^\circ_{\overline{k}},\bb{Q}_\ell(1))$. Let $\Lambda_{X,\cal{L},\bb{Q}_\ell}^\perp$ be the orthogonal complement of $\Lambda_{X,\cal{L},\bb{Q}_\ell}$ with respect to the cup product. Since the intersection pairing matrix of $\Lambda_{X,\cal{L},\bb{Q}_{\ell}}$ is negative definite, $H^2_{\et}(X^\circ_{\overline{k}},\bb{Q}_{\ell}(1))\cong \Lambda_{X,\cal{L},\bb{Q}_\ell}\oplus \Lambda_{X,\cal{L},\bb{Q}_\ell}^\perp$.
    
    We extend the action of $\cal{W}_{X,\cal{L}}$ on $\Lambda_{X,\cal{L},\bb{Q}_\ell}$ to $H^2_{\et}(X^\circ_{\overline{k}},\bb{Q}_{\ell}(1))$ by letting $\cal{W}_{X,\cal{L}}$ act trivially on $\Lambda_{X,\cal{L},\bb{Q}_\ell}^\perp$. This defines $i_{\ell}:\cal{W}_{X,\cal{L}}\to \Aut_{G_k}(H^2_{\et}(X^\circ_{\overline{k}},\bb{Q}_{\ell}(1)))$. We need to check that it has the desired properties:
    \begin{enumerate}
        \item We first check that the image of $i_\ell$ lies in $\Aut_{G_k}(H^2_{\et}(X^\circ_{\overline{k}},\bb{Q}_\ell(1)))$ and $i_\ell$ is $G$-equivariant. Both $G$ and $G_k$ preserve the decomposition $H^2_{\et}(X^\circ_{\overline{k}},\bb{Q}_{\ell}(1))\cong \Lambda_{X,\cal{L},\bb{Q}_\ell}\oplus \Lambda_{X,\cal{L},\bb{Q}_\ell}^\perp$, and the image of $i_\ell$ is trivial on $\Lambda_{X,\cal{L},\bb{Q}_\ell}^\perp$. It therefore remains to check both over $\Lambda_{X,\cal{L},\bb{Q}_\ell}$, where $i_\ell$ is defined by the action of $\cal{W}_{X,\cal{L}}$ on $\Lambda_{X,\cal{L}}$, which is constructed over $k$, together with the cycle class map, which is canonical.
        \item It remains to check that $i_\ell(s_{E_i})(\alpha)=\alpha+\frac{1}{n_i}([E_i]\cup\alpha)[E_i]$. Write $\alpha=[E']+\alpha'$ where $E'\in\Lambda_{X,\cal{L},\bb{Q}_\ell}$ and $\alpha'\in\Lambda_{X,\cal{L},\bb{Q}_\ell}^\perp$. Then we have
        \[
            i_\ell(s_{E_i})(\alpha)=[E']+\frac{1}{n_i}(E_i.E')[E_i]+\alpha'=\alpha+\frac{1}{n_i}(E_i.E')[E_i]=\alpha+\frac{1}{n_i}([E_i]\cup\alpha)[E_i].
        \]
    \end{enumerate}
\end{proof}

Let $X$ and $Y$ be two K3 or Enriques surfaces over $K$ isomorphic to each other. Suppose they have good reduction. Let $f:\cal{X}\dashrightarrow\cal{Y}$ be an $\cal{O}_K$-linear birational map between smooth models of $X$ and $Y$ that restricts to a $K$-linear isomorphism on the generic fibre.

\begin{prop}[cf. {\cite[Theorem 8.4]{Chiarellotto_2019_neron_ogg_shafarevich_criterion_k3surface}}]\label{8.4}
    Let $\cal{L}$ be a polarization of $X$. Suppose $f$ is $\cal{L}$-terminal. Then
    \[
        i_{\ell}(s_f)=s_{f,\ell}.
    \]
\end{prop}
\begin{proof}
    We first prove that $i_\ell(s_f)$ and $s_{f,\ell}$ agree on the subspace $\Lambda_{X,\cal{L},\bb{Q}_\ell}$. By the commutative diagram \ref{l-adic analogue}, for any $D\in \Lambda_{X,\cal{L},\bb{Q}_\ell}$, $s_{f,\ell}([D])=[s_f(D)]$. We need to show that
    \[
        [s_f(D)]=i_\ell(s_f)([D]).
    \]
    Since $s_f\in\cal{W}_{X,\cal{L}}$, it suffices to check this for $s_f=s_E$. This is obvious by the construction of $i_\ell$.

    It remains to prove that $i_\ell(s_f)$ and $s_{f,\ell}$ agree on $\Lambda^\perp_{X,\cal{L},\bb{Q}_\ell}$. Since $f$ is $\cal{L}$-terminal, the special fibre of the graph of $f$ decomposes as
    \[
        \Gamma_{f,k} = \Gamma_{f_k} + \sum_{i,j} b_{ij}(E_i \times E_j')
    \]
    for integers $b_{ij} \geq 0$, where the $E_i$ are the irreducible components of $E_{X,\cal{L}}$ and the $E_j'$ those of $E_{Y,f_*\cal{L}}$. Then for every $\alpha \in H^2_{\et}(X^\circ_{\bar{k}}, \bb{Q}_\ell(1))$,
    \[
        s_{f,\ell}(\alpha) = \alpha + \sum_{i,j} b_{ij}\big([E_j'] \cup (f_k^*)^{-1}\alpha\big)[E_i].
    \]
    Let $\alpha\in \Lambda_{X,\cal{L},\bb{Q}_\ell}^\perp$. Since $f_k : X^\circ \to Y^\circ$ maps $E_{X,\cal{L}}$ to $E_{Y,f_*\cal{L}}$, $[E_j']\cup (f^*_k)^{-1}\alpha=0$ and $s_{f,\ell}(\alpha)=\alpha$.
    
    It suffices to prove that $i_\ell(s_f)$ is trivial on $\Lambda^\perp_{X,\cal{L},\bb{Q}_\ell}$. Since $s_f \in W_{X,\cal{L}}$, it suffices to check this for $s_f = s_E$. This holds by the construction of $i_\ell$.
\end{proof}

\subsection{Proof of Theorem \ref{main theorem}}

Define $j_\ell$ as the composition of $i_\ell$ with the restriction $\Aut_{G_k}(H^2_{\et}(X^\circ_{\overline{k}},\bb{Q}_{\ell}(1)))\to\Aut(H^2_{\et}(X^\circ_{\overline{k}},\bb{Q}_\ell(1))^{G_k})$, $$j_\ell:\cal{W}_{X,\cal{L}}\xrightarrow{i_\ell}\Aut_{G_k}(H^2_{\et}(X^\circ_{\overline{k}},\bb{Q}_{\ell}(1)))\to\Aut(H^2_{\et}(X^\circ_{\overline{k}},\bb{Q}_\ell(1))^{G_k}).$$
Then $j_\ell$ induces a pointed map
$${j_\ell}_*:H^1(G,\cal{W}_{X,\cal{L}})\to H^1(G,\Aut(H^2_{\et}(X^\circ_{\overline{k}},\bb{Q}_\ell(1))^{G_k})).$$ Let $\beta_{\cal{L},\ell}={j_\ell}_*(\alpha_\cal{L})$.

\begin{lem}\label{trivial kernel}
    The pointed map ${j_{\ell}}_*$ has trivial kernel.
\end{lem}
\begin{proof}
    Denote the action of $\cal{W}_{X,\cal{L}}$ on $\Lambda_{X,\cal{L},\bb{Q}_{\ell}}$ by $j'_\ell:\cal{W}_{X,\cal{L}}\to\Aut(\Lambda_{X,\cal{L},\bb{Q}_{\ell}})$, so we have a pointed map $${j'_\ell}_*:H^1(G,\mathcal{W}_{X,\cal{L}})\to H^1(G,\Aut(\Lambda_{X,\cal{L},\bb{Q}_{\ell}})).$$ By \cite[Theorem 4.1]{Chiarellotto_2019_neron_ogg_shafarevich_criterion_k3surface}, ${j'_\ell}_*$ has trivial kernel, so it suffices to prove that ${j_\ell}_*$ and ${j'_\ell}_*$ have the same kernel.

    The cycle class map gives a $G$-injection $\Lambda_{X,\cal{L},\bb{Q}_{\ell}}\to H^2_{\et}(X^\circ_{\overline{k}},\bb{Q}_\ell(1))^{G_k}$. Denote by $T_{\ell}$ the quotient and consider the short exact sequence of $G$-representations $$0\to\Lambda_{X,\cal{L},\bb{Q}_{\ell}}\to H^2_{\et}(X^\circ_{\overline{k}},\bb{Q}_\ell(1))^{G_k}\to T_{\ell}\to0.$$ Suppose $\alpha\in H^1(G,\cal{W}_{X,\cal{L}})$. By Lemma \ref{twist a module}, $${j'_\ell}_*(\alpha)\text{ is trivial}\iff \Lambda_{X,\cal{L},\bb{Q}_{\ell}}\cong\Lambda_{X,\cal{L},\bb{Q}_{\ell}}^\alpha,$$ $${j_\ell}_*(\alpha)\text{ is trivial}\iff H^2_{\et}(X^\circ_{\overline{k}},\bb{Q}_\ell(1))^{G_k}\cong H^2_{\et}(X^\circ_{\overline{k}},\bb{Q}_\ell(1))^{G_k,\alpha}.$$ Next, by the construction of the action of $\cal{W}_{X,\cal{L}}$ on $H^2_{\et}(X^\circ_{\overline{k}},\bb{Q}_\ell(1))$, $\cal{W}_{X,\cal{L}}$ acts trivially on $T_{\ell}$, so $T_{\ell}\cong T_{\ell}^\alpha$. Finally, by the semi-simplicity of the representations of the finite group $G$, we have $$\Lambda_{X,\cal{L},\bb{Q}_{\ell}}\cong\Lambda_{X,\cal{L},\bb{Q}_{\ell}}^\alpha\iff H^2_{\et}(X^\circ_{\overline{k}},\bb{Q}_\ell(1))^{G_k}\cong H^2_{\et}(X^\circ_{\overline{k}},\bb{Q}_\ell(1))^{G_k,\alpha}.$$ This finishes the proof.
\end{proof}

In particular, $\beta_{\cal{L},\ell}$ controls the extendability of the $G$-action by Theorem \ref{first criterion}. Let $H^2_{\et}(X^\circ_{\overline{k}},\bb{Q}_\ell(1))^{G_k,\beta_{\cal{L},\ell}}$ be the $G$-representation obtained as the twist of $H^2_{\et}(X^\circ_{\overline{k}},\bb{Q}_\ell(1))^{G_k}$ by $\beta_{\cal{L},\ell}$.

\begin{prop}[cf. {\cite[Theorem 8.6]{Chiarellotto_2019_neron_ogg_shafarevich_criterion_k3surface}}]\label{semi-linear case twist}
    There is a $G$-isomorphism 
    \[
        H^2_{\et}(X_{K^{\mathrm{s}}},\bb{Q}_\ell(1))^{G_k}\cong H^2_{\et}(X^\circ_{\overline{k}},\bb{Q}_\ell(1))^{G_k,\beta_{\cal{L},\ell}}.
    \]
\end{prop}
\begin{proof}
    Since $X$ has good reduction over $K$, we have a $G_k$-equivariant isomorphism
    \[
        \phi:H^2_{\et}(X^\circ_{\overline{k}},\bb{Q}_\ell(1))\to H^2_{\et}(X_{K^{\mathrm{s}}},\bb{Q}_\ell(1)),
    \]
    which restricts to an isomorphism $H^2_{\et}(X^\circ_{\overline{k}},\bb{Q}_\ell(1))^{G_k}\to H^2_{\et}(X_{K^{\mathrm{s}}},\bb{Q}_\ell(1))^{G_k}$. We claim that this gives a $G$-isomorphism if we twist $H^2_{\et}(X^\circ_{\overline{k}},\bb{Q}_\ell(1))^{G_k}$ by $\beta_{\cal{L},\ell}$. Let $g\in G$ and let $\overline{g}$ be any lift of $g$ to $\overline{G}_k$. Let $\cal{X}$ be an $\cal{L}$-terminal model of $X$. Then $\alpha_{\cal{L},\cal{X}}$ is a cocycle representing $\alpha_{\cal{L}}$. Let $\beta_{\cal{L},\ell,\cal{X}}=j_\ell\circ \alpha_{\cal{L},\cal{X}}$, which is a cocycle representing $\beta_{\cal{L},\ell}$. It suffices to show that on $H^2_{\et}(X^\circ_{\overline{k}},\bb{Q}_\ell(1))^{G_k}$,
    \[
        \phi\circ\beta_{\cal{L},\ell,\cal{X}}(g)\circ \overline{g}_k^*=\overline{g}^*\circ\phi.
    \]
    By Lemma \ref{generic pullback} and Proposition \ref{8.4},
    \[
        \beta_{\cal{L},\ell,\cal{X}}(g)=(j_\ell\circ\alpha_{\cal{L},\cal{X}})(g)=i_\ell(s_{g^{\#}_{\cal{X}}})=s_{g^{\#}_{\cal{X}},\ell}=\phi^{-1}\circ (g^{\#})^*\circ\phi\circ((g^{\#}_k)^*)^{-1}.
    \]
    It suffices to prove that
    \begin{equation}\label{ladic group action}
        \overline{g}^*\circ\phi\circ (\overline{g}_k^*)^{-1}=(g^{\#})^*\circ\phi\circ((g^{\#}_k)^*)^{-1}.
    \end{equation}

    Let $\overline{\sigma}_g$ be any lift of $\sigma_g\in G_{k/k_0}$ to $G_{k_0}$, so $(g,\overline{\sigma}_g)$ lifts $g$ to $\overline{G}_k$. Therefore, $\overline{g}*=(g^\#)^*\circ\overline{\sigma}_g^*$ and $\overline{g}_k^*=(g_k^\#)\circ\overline{\sigma}_g^*$. Then the LHS and the RHS of \ref{ladic group action} fit in the following diagram
    \[
        \xymatrix{H^2_{\et}(X^\circ_{\overline{k}},\bb{Q}_{\ell}(1)) \ar@{->}[r]^{((g^{\#})^*)^{-1}} & H^2_{\et}((X^{\circ}_{\overline{k}})^{\overline{\sigma}_g},\bb{Q}_{\ell}(1)) \ar@{->}[r]^{(\overline{\sigma}_g^*)^{-1}} \ar@{->}[d]_{\phi} & H^2_{\et}(X^\circ_{\overline{k}},\bb{Q}_{\ell}(1)) \ar@{->}[d]_{\phi} \\ H^2_{\et}(X_{K^s},\bb{Q}_{\ell}(1)) & H^2_{\et}(X_{K^s}^{\overline{\sigma}_g},\bb{Q}_{\ell}(1)) \ar@{->}[l]_{(g^{\#})^*} & H^2_{\et}(X_{K^s},\bb{Q}_{\ell}(1)) \ar@{->}[l]_{\overline{\sigma}_g^*}}
    \]
    To prove \ref{ladic group action}, it suffices to show that the square commutes, which is obvious.
\end{proof}

Now we obtain the following theorem.

\begin{thm}\label{refined main theorem}
    Let $X$ be a K3 or Enriques surface over $K$ with good reduction. Let $G$ be a finite group acting on $X$ with base field $K_0$. Assume $K/K_0$ is finite and unramified. Then the following are equivalent.
    \begin{enumerate}
        \item The $G$-action on $X$ is extendable.
        \item There is a $\overline{G}_k$-isomorphism $$H^2_{\et}(X_{K^{\mathrm{s}}},\bb{Q}_\ell(1))\cong H^2_{\et}(X^\circ_{\overline{k}},\bb{Q}_\ell(1)).$$
        \item There is a $G$-isomorphism $$H^2_{\et}(X_{K^{\mathrm{s}}},\bb{Q}_\ell(1))^{G_k}\cong H^2_{\et}(X^\circ_{\overline{k}},\bb{Q}_\ell(1))^{G_k}.$$
    \end{enumerate}
\end{thm}

\begin{proof}
    Suppose the $G$-action on $X$ extends to a smooth model $\cal{X}$ of $X$. The comparison isomorphism $H^2_{\et}(X_{K^{\mathrm{s}}},\bb{Q}_\ell(1))\cong H^2_{\et}(X^\circ_{\overline{k}},\bb{Q}_\ell(1))$ given by $\cal{X}$ is equivariant with respect to the action of $\overline{G}_k$. Therefore, $(1)\implies (2)$.
    
    Also, the restriction of a $\overline{G}_k$-isomorphism $H^2_{\et}(X_{K^{\mathrm{s}}},\bb{Q}_\ell(1))\cong H^2_{\et}(X^\circ_{\overline{k}},\bb{Q}_\ell(1))$ to the $G_k$-invariant subspace is exactly a $G$-isomorphism $H^2_{\et}(X_{K^{\mathrm{s}}},\bb{Q}_\ell(1))^{G_k}\cong H^2_{\et}(X^\circ_{\overline{k}},\bb{Q}_\ell(1))^{G_k}$. Therefore, $(2)\implies (3)$.
    
    Lastly, let $\cal{L}$ be a polarization on $X$ fixed by $G$. We have
    \begin{align*}
        & H^2_{\et}(X_{\overline{k}}^\circ,\bb{Q}_\ell(1))^{G_k}\cong H^2_{\et}(X_{K^{\mathrm{s}}},\bb{Q}_\ell(1))^{G_k}\text{ as $G$-representations} \\ \iff & \beta_{\cal{L},\ell}\text{ is trivial} \iff \alpha_\cal{L} \text{ is trivial} \iff \text{the $G$-action is extendable}
    \end{align*}
    by Lemma \ref{twist a module}, Proposition \ref{semi-linear case twist}, Lemma \ref{trivial kernel}, and Theorem \ref{first criterion}. Thus, $(1)\iff (3)$. This finishes the proof.
\end{proof}

In particular, this proves Theorem \ref{main theorem}.

\subsection{Proof of Theorem \ref{linear case main theorem}}

The proof of Theorem \ref{linear case main theorem} is completely analogous to that of Theorem \ref{refined main theorem}. Suppose the $G$-action on $X$ is $K$-linear. Then $G$ naturally acts on $X^\circ_{\overline{k}}$ and $\Aut_{\bb{Q}_\ell}(H^2_{\et}(X^\circ_{\overline{k}},\bb{Q}_\ell(1)))$.

Define $q_\ell$ as the composition of $i_\ell$ with the inclusion $\Aut_{G_k}(H^2_{\et}(X^\circ_{\overline{k}},\bb{Q}_{\ell}(1)))\to\Aut_{\bb{Q}_\ell}(H^2_{\et}(X^\circ_{\overline{k}},\bb{Q}_\ell(1)))$,
$$q_\ell:\cal{W}_{X,\cal{L}}\xrightarrow{i_\ell}\Aut_{G_k}(H^2_{\et}(X^\circ_{\overline{k}},\bb{Q}_{\ell}(1)))\to\Aut_{\bb{Q}_\ell}(H^2_{\et}(X^\circ_{\overline{k}},\bb{Q}_\ell(1))).$$
Then $q_\ell$ induces a pointed map
$${q_\ell}_*:H^1(G,\cal{W}_{X,\cal{L}})\to H^1(G,\Aut_{\bb{Q}_\ell}(H^2_{\et}(X^\circ_{\overline{k}},\bb{Q}_\ell(1)))).$$
Let $\gamma_{\cal{L},\ell}={q_\ell}_*(\alpha_\cal{L})$.

\begin{lem}\label{trivial kernel linear case}
    The pointed map ${q_{\ell}}_*$ has trivial kernel.
\end{lem}
\begin{proof}
    Denote the action of $\cal{W}_{X,\cal{L}}$ on $\Lambda_{X,\cal{L},\bb{Q}_{\ell}}$ by $j'_\ell:\cal{W}_{X,\cal{L}}\to\Aut(\Lambda_{X,\cal{L},\bb{Q}_{\ell}})$. The kernel of ${j'_{\ell}}_*$ is trivial by \cite[Theorem 4.1]{Chiarellotto_2019_neron_ogg_shafarevich_criterion_k3surface}. It suffices to prove that ${q_\ell}_*$ and ${j'_\ell}_*$ have the same kernel.

    The cycle class map gives a $G$-injection 
    $\Lambda_{X,\cal{L},\bb{Q}_{\ell}}\to H^2_{\et}(X^\circ_{\overline{k}},\bb{Q}_\ell(1))$. Denote by $S_{\ell}$ the quotient and consider the short exact sequence of $G$-representations $$0\to\Lambda_{X,\cal{L},\bb{Q}_{\ell}}\to H^2_{\et}(X^\circ_{\overline{k}},\bb{Q}_\ell(1))\to S_{\ell}\to0.$$ Suppose $\alpha\in H^1(G,\cal{W}_{X,\cal{L}})$. By Lemma \ref{twist a module},
    \[
        {j'_\ell}_*(\alpha)\text{ is trivial}\iff \Lambda_{X,\cal{L},\bb{Q}_{\ell}}\cong\Lambda_{X,\cal{L},\bb{Q}_{\ell}}^\alpha,
    \]
    \[
        {q_\ell}_*(\alpha)\text{ is trivial}\iff H^2_{\et}(X^\circ_{\overline{k}},\bb{Q}_\ell(1))\cong H^2_{\et}(X^\circ_{\overline{k}},\bb{Q}_\ell(1))^\alpha.
    \]
    Next, $\cal{W}_{X,\cal{L}}$ acts trivially on $S_{\ell}$ by construction, so $S_{\ell}\cong S_{\ell}^\alpha$. Then we are done by the semi-simplicity of the representations of the finite group $G$.
\end{proof}

Let $H^2_{\et}(X^\circ_{\overline{k}},\bb{Q}_\ell(1))^{\gamma_{\cal{L},\ell}}$ be the twist of $H^2_{\et}(X^\circ_{\overline{k}},\bb{Q}_\ell(1))$ by $\gamma_{\cal{L},\ell}$.

\begin{prop}\label{linear case twist}
    There is a $G$-isomorphism $$H^2_{\et}(X_{K^{\mathrm{s}}},\bb{Q}_\ell(1))\cong H^2_{\et}(X^\circ_{\overline{k}},\bb{Q}_\ell(1))^{\gamma_{\cal{L},\ell}}.$$
\end{prop}
\begin{proof}
    Since $X$ has good reduction over $K$, we have an isomorphism
    \[
        \phi:H^2_{\et}(X^\circ_{\overline{k}},\bb{Q}_\ell(1))\to H^2_{\et}(X_{K^{\mathrm{s}}},\bb{Q}_\ell(1)).
    \]
    We claim that this gives a $G$-isomorphism if we twist $H^2_{\et}(X^\circ_{\overline{k}},\bb{Q}_\ell(1))$ by $\gamma_{\cal{L},\ell}$. Let $\cal{X}$ be an $\cal{L}$-terminal model of $X$ and let $\gamma_{\cal{L},\ell,\cal{X}}=q_\ell\circ \alpha_{\cal{L},\cal{X}}$, which is a cocycle representing $\gamma_{\cal{L},\ell}$. Then it suffices to show that
    \[
        \phi\circ\gamma_{\cal{L},\ell,\cal{X}}(g)\circ g_k^*=g^*\circ\phi.
    \]
    This is much simpler than the proof of Proposition \ref{semi-linear case twist} since $g=g^\#$. By Lemma \ref{generic pullback} and Proposition \ref{8.4}, 
    \[
        \gamma_{\cal{L},\ell,\cal{X}}(g)=(q_\ell\circ\alpha_{\cal{L},\cal{X}})(g)=i_\ell(s_{g_{\cal{X}}})=s_{g_{\cal{X}},\ell}=\phi^{-1}\circ g^*\circ\phi\circ((g_k)^*)^{-1}.
    \]
    This finishes the proof.
\end{proof}

\begin{proof}[Proof of Theorem \ref{linear case main theorem}]
    Let $\cal{L}$ be a polarization on $X$ fixed by $G$. We have
    \begin{align*}
        & H^2_{\et}(X_{\overline{k}}^\circ,\bb{Q}_\ell)\cong H^2_{\et}(X_{K^{\text{s}}},\bb{Q}_\ell)\text{ as $G$-representations} \\ \iff & H^2_{\et}(X_{\overline{k}}^\circ,\bb{Q}_\ell(1))\cong H^2_{\et}(X_{K^{\text{s}}},\bb{Q}_\ell(1))\text{ as $G$-representations} \\ \iff & \gamma_{\cal{L},\ell}\text{ is trivial} \iff \alpha_\cal{L} \text{ is trivial} \iff \text{the $G$-action is extendable}
    \end{align*}
    by Lemma \ref{twist a module}, Proposition \ref{linear case twist}, Lemma \ref{trivial kernel linear case}, and Theorem \ref{first criterion}.
\end{proof}

\section{Applications}\label{application}

In this section, we prove Corollary \ref{characteristic and extendability} and Corollary \ref{good reduction of enriques surfaces}.

\subsection{Extendability of symplectic group actions}

We prove Corollary \ref{characteristic and extendability} in this sub-section. Suppose $X$ is a K3 surface over $K$ with good reduction and $G$ is a finite group acting $K$-linearly on $X$.

\begin{lem}\label{order preserving}
    If $p\nmid |G|$, for every $g\in G$, $g$ and $g_k$ have the same order.
\end{lem}
\begin{proof}
    Suppose $g$ has order $n$ and $g_k$ has order $m$. The order of $g_k$ divides the order of $g$ so $n=mk$ for some integer $k$. By \cite[Proposition 6.1]{matsumoto_2021_extendability_automorphisms_k3surfaces}, $k$ is a power of $p$. Since $p\nmid |G|$, we have $k=1$ so $m=n$ as we want.
\end{proof}

\begin{lem}\label{symplecticness preserving}
    If $g$ is symplectic on $X$, so is $g_k$ on $X^\circ$.
\end{lem}
\begin{proof}
    By \cite[Proposition 4.7]{liedtke_2017_good_reduction_k3surfaces}, there exists a smooth model $\cal{X}$ such that $g$ is defined over an open dense subset of $U\subset\cal{X}$ whose complement consists of finitely many curves on the special fibres. Denote the intersection of $U$ with $X^\circ$ as $V$, an open dense subset of $X^\circ$. Denote the relative canonical divisor of $U$ and $V$ respectively as $\omega_U$ and $\omega_V$.

    We make the following claims.
    \begin{enumerate}
        \item The restriction map $H^0(X^\circ,\omega_{X^\circ})\to H^0(V,\omega_V)$ is injective.
        \item The restriction map $H^0(U,\omega_U)\to H^0(X,\omega_X)$ is injective.
        \item The image of the restriction map $H^0(U,\omega_U)\to H^0(V,\omega_V)$ equals the image of $H^0(X^\circ,\omega_{X^\circ})\to H^0(V,\omega_V)$.
    \end{enumerate}
    Indeed, since $X^\circ$ is integral and $\omega_{X^\circ}$ is trivial, (1) holds. Also, $U$ is isomorphic to $\cal{X}$ away from codimension $2$, so $H^0(U,\omega_U)=H^0(\cal{X},\omega_{\cal{X}})=\cal{O}_K$, which is torsion-free. Since the restriction map in (2) is the base change from $\cal{O}_K$ to $K$, it is injective. Finally, we have a commutative diagram
    $$\xymatrix{ H^0(\cal{X},\omega_{\cal{X}}) \ar@{->}[r] \ar@{->}[d] & H^0(X^\circ,\omega_{X^\circ}) \ar@{->}[d] \\ H^0(U,\omega_U) \ar@{->}[r] & H^0(V,\omega_V)}  $$ The vertical map on the left is an isomorphism. The horizontal map on the top corresponds to the quotient homomorphism $\cal{O}_K\to k$, therefore surjective. Then (3) holds by diagram chasing.

    Suppose $g$ acts trivially on $H^0(X,\omega_X)$. By (2), $g$ also acts trivially on $H^0(U,\omega_U)$ and thus on the image of $H^0(U,\omega_U)\to H^0(V,\omega_V)$. Then by (3), $g_k$ acts trivially on the image of $H^0(X^\circ,\omega_{X^\circ})\to H^0(V,\omega_V)$. Finally, by (1), $g_k$ acts trivially on $H^0(X^\circ,\omega_{X^\circ})$. Therefore, $g_k$ is symplectic on $X^\circ$.
\end{proof}

\begin{proof}[Proof of Corollary \ref{characteristic and extendability}]
    Let $\ell\neq p$ be a prime number. By Theorem \ref{linear case main theorem}, we need to prove that $H^2_{\et}(X_{K^{\mathrm{s}}},\bb{Q}_\ell)$ and $H^2_{\et}(X^\circ_{\overline{k}},\bb{Q}_\ell)$ are isomorphic as $G$-representations. Since $\bb{Q}_\ell$ has characteristic $0$, it suffices to compare their characters. Equivalently, we prove that for all $g\in G$, $g$ and $g_k$ have the same characteristic polynomial on $H^2_{\et}(X_{K^{\mathrm{s}}},\bb{Q}_\ell)$ and $H^2_{\et}(X^\circ_{\overline{k}},\bb{Q}_\ell)$ respectively.
    
    First, $g$ and therefore $g_k$ are symplectic by Lemma \ref{symplecticness preserving}. Then by \cite[Lemma 2.13]{matsumoto_2021_extendability_automorphisms_k3surfaces}, the characteristic polynomials of $g$ and $g_k$ only depend on their orders. Finally, by Lemma \ref{order preserving}, $g$ and $g_k$ have the same order. This finishes the proof.
\end{proof}

\subsection{Good reduction of Enriques surfaces}

We prove Corollary \ref{good reduction of enriques surfaces} in this sub-section. First, we recall some facts about Enriques surfaces.

Suppose the ground field $F$ has characteristic different from $2$. Then the quotient of every K3 surface over $F$ by an involution without fixed points is an Enriques surface by \cite[Proposition 10.15]{algebraic_surface}.

Conversely, every Enriques surface over $F$ is the quotient of a K3 double cover by a fixed-point-free involution. The construction is as follows. Suppose $Y$ is an Enriques surface. We choose an isomorphism $\omega_Y^{\otimes 2}\cong\cal{O}_Y$ and form the $\cal{O}_Y$-algebra $\cal{O}_Y\oplus\omega_Y$. We take the K3 double cover to be its relative spectrum $$X=\underline{\Spec}_Y(\cal{O}_Y\oplus\omega_Y)$$ and the associated involution $\iota$ is the unique involution that acts as the identity on $\cal{O}_Y$ and the $(-1)$-map on $\omega_Y$. This is a K3 surface by the proof of \cite[Proposition 10.14]{algebraic_surface}.

In fact, this construction can be done in families.

\begin{lem}\label{k3 double cover in family}
    Assume $p\neq 2$. Let $Y$ be an Enriques surface over $K$ with good reduction. Let $\cal{Y}$ be a smooth model of $Y$ over $\cal{O}_K$. Then there exists a smooth proper algebraic space $\cal{X}$ over $\cal{O}_K$ with an involution $\iota$ such that $\cal{X}_K$ and $\cal{X}_k$ are respectively K3 double covers of $Y$ and $\cal{Y}_k$ with associated fixed-point-free involutions $\iota_K$ and $\iota_k$.
\end{lem}
\begin{proof}
    We mimic the construction of a K3 double cover of an Enriques surface. Since $Y$ is an Enriques surface, $\omega_{Y}$ is nontrivial but $\omega_{Y}^{\otimes 2}\cong \cal{O}_{Y}$. The pullback induces an isomorphism $\Pic(\cal{Y})\to\Pic(Y)$ that maps $\omega_{\cal{Y}}$ to $\omega_{Y}$. Therefore, $\omega_{\cal{Y}}$ is nontrivial but $\omega_{\cal{Y}}^{\otimes 2}\cong \cal{O}_{\cal{Y}}$. We choose such an isomorphism and construct the $\cal{O}_{\cal{Y}}$-algebra $\cal{O}_{\cal{Y}}\oplus\omega_{\cal{Y}}$. Let $$\cal{X}=\underline{\Spec}_{\cal{Y}}(\cal{O}_{\cal{Y}}\oplus\omega_{\cal{Y}})$$ and $\iota$ be the unique involution that acts as the identity on $\cal{O}_{\cal{Y}}$ and the $(-1)$-map on $\omega_{\cal{Y}}$. Then $\cal{X}/\iota\cong\cal{Y}$.

    Since $p\neq 2$, the unramified double cover $\cal{X}\to \cal{X}/\iota\cong \cal{Y}$ is finite \'etale. In particular, $\cal{X}$ is smooth and proper over $\cal{O}_K$. Since taking the relative spectrum commutes with base change, $\cal{X}_K$ and $\cal{X}_k$ are respectively K3 double covers of $Y$ and $\cal{Y}_k$ with associated fixed-point-free involutions $\iota_K$ and $\iota_k$.
\end{proof}

Now we prove Corollary \ref{good reduction of enriques surfaces}.

\begin{proof}[Proof of Corollary \ref{good reduction of enriques surfaces}]
    Suppose $Y$ has good reduction with a smooth model $\cal{Y}$. Apply Lemma \ref{k3 double cover in family} to $Y$ and $\cal{Y}$. Let $\cal{X}$ and $\iota$ be as in Lemma \ref{k3 double cover in family}. Then $\cal{X}_K$ is a K3 double cover of $Y$ with associated involution $\iota_K$, whose specialization $\iota_k$ has no fixed points. This proves $(1)\implies (2)$. 
    
    Also, suppose $X$ is a K3 double cover of $Y$, with associated involution $\iota$, such that $X$ has good reduction and $\iota_k$ has no fixed points. Then $X^\circ/\iota_k$ is an Enriques surface. By \cite[Section 10.10]{algebraic_surface},
    $$\dim H^2_{\et}(X^\circ_{\overline{k}},\bb{Q}_\ell)^{\iota_k=\mathrm{id}}=\dim H^2_{\et}((X^\circ/\iota_k)_{\overline{k}},\bb{Q}_\ell)=10.$$
    This proves that $(2)\implies (3)$.

    It remains to prove that $(3)\implies (1)$. Suppose $X$ is a K3 double cover of $Y$, with associated involution $\iota$, such that $X$ has good reduction and $\dim H^2_{\et}(X^\circ_{\overline{k}},\bb{Q}_\ell)^{\iota_k=\mathrm{id}}=10.$ We claim that $\iota$ is extendable. Indeed, let $G=\langle\iota\rangle\cong\bb{Z}/2$. By Theorem \ref{linear case main theorem}, it suffices to prove that there is a $G$-isomorphism $$H^2_{\et}(X_{K^{\mathrm{s}}},\bb{Q}_\ell)\cong H^2_{\et}(X^\circ_{\overline{k}},\bb{Q}_\ell).$$ Every finite dimensional representation of $G$ decomposes as the direct sum of the subspace fixed by $\iota$ and the subspace where $\iota$ acts as the $(-1)$-map. Since $H^2_{\et}(X_{K^{\mathrm{s}}},\bb{Q}_\ell)$ and $ H^2_{\et}(X^\circ_{\overline{k}},\bb{Q}_\ell)$ both have dimension 22, it suffices to show that
    \begin{equation}\label{dim argument}
        \dim H^2_{\et}(X_{K^{\mathrm{s}}},\bb{Q}_\ell)^{\iota=\mathrm{id}}=\dim H^2_{\et}(X^\circ_{\overline{k}},\bb{Q}_\ell)^{\iota_k=\mathrm{id}}.
    \end{equation}
    Since $Y$ is an Enriques surface, by \cite[Section 10.10]{algebraic_surface}, the LHS of \ref{dim argument} is $$\dim H^2_{\et}(X_{K^{\mathrm{s}}},\bb{Q}_\ell)^{\iota=\mathrm{id}} = \dim H^2_{\et}(Y_{K^{\mathrm{s}}},\bb{Q}_\ell)=10.$$ On the other hand, the RHS of \ref{dim argument} is $10$ by assumption. This shows that $\iota$ is extendable.
    
    Let $\cal{X}$ be a smooth model of $X$ to which $\iota$ extends. We want to show that $\cal{Y}=\cal{X}/\iota$ is a smooth model of $Y$. By construction, $\cal{Y}$ is a model of $Y$. We need to prove that $\cal{Y}$ is smooth. By flatness, it suffices to prove that $\cal{Y}_k$ is smooth.
    
    Since $p\neq 2$, we have $\cal{Y}_k\cong X^\circ/{\iota_k}$. We claim that $\iota_k$ is non-symplectic. Suppose $\iota_k$ is symplectic. By \cite[Theorem 5.1]{mu_n_action_of_k3_matsumoto}, $\cal{Y}_k$ is a surface with at worst RDPs whose minimal desingularization is a K3 surface. In particular, the Euler characteristic of $\cal{Y}_k$ is $2$. However, the Euler characteristic of $\cal{Y}_k$ should agree with that of $Y$, which is $1$. This is a contradiction, so $\iota_k$ is non-symplectic.
    
    If $\iota_k$ has no fixed points, $\cal{Y}_k$ is an Enriques surface, which is smooth. Therefore, we are left with the case that $\iota_k$ is non-symplectic and has a fixed point. Then we prove $\cal{Y}_k$ is smooth using a formal local computation: it suffices to prove this over the algebraic closure, so assume $k=\overline{k}$. Let $\Spec k[\![x,y]\!]$ be a formal local neighbourhood of a fixed point of $\iota_k$. By non-symplecticness, we may normalize the action $\iota_k$ on $k[\![x,y]\!]$ such that it sends $x\mapsto -x$, and $y\mapsto y$. Then the ring of invariants is $k[\![x^2,y]\!]$, whose spectrum is smooth. This shows that $(3)\implies (1)$ and finishes the proof.
\end{proof}

\printbibliography

\end{document}